\documentclass[10pt]{article}

\usepackage[colorlinks=true,
            linkcolor=black,
            citecolor=black,
            urlcolor=black,
            filecolor=black,
            menucolor=black,
            pdfborder={0 0 0}]{hyperref}

\usepackage[labelsep=endash]{caption}

\usepackage{graphicx}
\usepackage{latexsym,amssymb}
\usepackage{amsthm}
\usepackage{indentfirst}
\usepackage{amsmath}
\usepackage{color}
\usepackage{xcolor}

\usepackage[
  paperwidth=192mm,
  paperheight=262mm,
  vmargin={19mm,19mm},
  hmargin={13.7mm,13.7mm},
  headsep=12pt,
  footskip=12pt,
]{geometry}

\newtheorem{Theorem}{Theorem}[section]
\newtheorem*{Theorem A}{Theorem A}

\newtheorem*{Conj*}{Conjecture}

\newtheorem{Proposition}[Theorem]{Proposition}
\newtheorem{Lemma}[Theorem]{Lemma}

\newtheorem{Remark-numbered}[Theorem]{Remark}
\newtheorem{Remarks-numbered}[Theorem]{Remarks}

\newtheorem{Claim-numbered}{Claim}

\newcommand{\Sing}{\text{Sing}}
\newcommand{\Rep}{\text{Rep}}
\newcommand{\dmu}{\,{\rm d}\mu}
\newcommand{\norm}[1]{\|#1\|}

\newcommand{\ps}{\phi_s}

\begin{document}
\title{Topological and Metric Pressure for Singular Flows}
\author {Meijie Zhao, Xiao Wen\footnote{X.W. was partially supported by National Key R\&D Program of China (No. 2022YFA1005801).
}
}

\maketitle

\begin{abstract}
In this paper, we introduce the notions of rescaled metric pressure and rescaled topological pressure for flows by considering three types of rescaled Bowen balls, which take  the flow velocity and time reparametrization into account. This approach effectively eliminates the influence of singularities. It is demonstrated that defining both metric pressure and topological pressure via several distinct Bowen balls is equivalent. Furthermore, under the assumptions that $\log \|X(x)\|$ is integrable and that $\mu(\mathrm{Sing}(X))=0$, we prove Katok's formula of pressure. We establish a partial variational principle that relates the rescaled metric pressure and the rescaled topological pressure.\\


\end{abstract}

\section{Introduction}\label{sec:intro}
The notion of pressure in dynamical systems arises from the thermodynamic formalism developed in the 1970s, inspired by ideas from statistical mechanics. The central idea is to extend the concept of topological entropy, introduced by Adler-Konheim-McAndrew \(\cite{adler1965topological}\), by weighting orbit complexity with a potential function. This led to the definition of topological pressure, first systematically studied by Ruelle \(\cite{ruelle1982repellers}\) and Walters \(\cite{walters2000introduction}\), among others. In this setting, pressure generalizes entropy: when the potential is identically zero, one recovers topological entropy.
In parallel, a metric version of pressure was established, combining the classical Kolmogorov-Sinai entropy of an invariant probability measure with the space average of the potential. This formulation appears naturally in the variational principle, which states that the topological pressure equals the supremum of the metric pressures over all invariant measures. The principle, first proved in full generality by Goodman \cite{goodman1971relating} and later developed by Walters \(\cite{walters2000introduction}\), provides the foundation for the study of equilibrium states, i.e. invariant measures that maximize this quantity.
Since then, the pressure formalism has become a fundamental tool in smooth dynamics, symbolic dynamics, and ergodic theory, with wide-ranging applications to equilibrium states, dimension theory, and multifractal analysis.

Pressures of flows can be defined by the pressures of its time 1 map without considering the special characteristics of the flows. Bowen-Ruelle \cite{Bowen1975} first introduced the topological pressure of a flow using the notion of spanning sets. And then He-Yang-Gao \cite{he2004metrical} gave a definition of metric pressure by using the notion of spanning sets and proved the Katok's formula for non-singular flows. A recent work of Wang-Chen \(\cite{WangChen}\) dealt with the pressure of flows by introducing Bowen balls with time reparametrization. These authors also established a Katok formula for the pressure of non-singular flows in the ergodic case with some additional assumptions. In addition, they established a variational principle for the topological pressure of flows.

The idea of adjusting the scale of a neighborhood of a regular point via flow speed for a singular flow dates back to Liao's foundational work on standard equation systems \(\cite{liao1974standard}\cite{liao1996qualitative}\).  Gan-Yang \(\cite{gan2018morse}\) extracted the geometric essence from Liao's work, developing it into a pivotal tool. Using this tool, Wen-Wen \(\cite{wen2019rescaled}\) introduced the notion of rescaled expansiveness and proved its equivalence with classical expansiveness for nonsingular flows.
By considering rescaled
Bowen balls with time reparametrizations, Wang-Wen \(\cite{WANG2024128673}\)  defined three new entropies for singular flows, showed the equivalence of the three new definitions for any
\(C^1\) vector  fields, and referred to them as rescaled metric entropy.

In this paper, we define new metric pressures and topological pressures for flows. The novelty is the introduction of three types of rescaled Bowen balls involving flow velocity and reparametrization of time to define pressure, which eliminate the influence of singularities. The
three new definitions of metric pressure and topological pressure are shown to be equivalent.   Furthermore, we establish the partial variational principle between the rescaled metric pressure and the rescaled topological pressure. When \( \log\|X(x)\| \) is integrable and \(\mu({\rm Sing}(X))=0\), we prove Katok's formula of pressure.
Let us present our results in a precise way.

Let $M$ be a compact metric space and $\phi_t$ be a continuous flow on $M$. Denote by $C(M, \mathbb{R})$ the Banach space of all continuous function on $M$ with the $C^0$ norm $\|\cdot\|$. A Borel probability measure $\mu$ on $M$ is called $\phi$ {\it invariant} if $\mu(\phi_t(A))=\mu(A)$ for every measurable set $A$ and $t\in\mathbb{R}$. $\mu$ is called $\phi$ {\it ergodic }if $\mu(A)=0 \text{ or } 1$ for every $\phi$ invariant set $A$. For \(x \in M\), \(t > 0\), \(\varepsilon > 0\), define
\[
B(x, t, \varepsilon) = \left\{ y \in M : d(\phi_{s}x, \phi_s y) < \varepsilon, \,\forall\, 0 \leq s \leq t \right\},
\]
and \(B(x, t, \varepsilon)\) is called the \((t, \varepsilon)\)-ball centering at \(x\).
For $F\subset M$, we say that $F$ $(t, \varepsilon)$ {\it spans} $M$ if
\[M\subset\bigcup_{x\in F}B(x, t, \varepsilon).\]
{\it Topological pressure} for the flow $\phi_t$ w.r.t a continuous function $f\in C(M, \mathbb{R})$ can be defined as
\[P(\phi, f)=\lim_{\varepsilon\to0}\limsup_{t\to\infty}\frac{1}{t}\log N(t, \varepsilon, f),\]
where
\[N(t, \varepsilon, f)=\inf\left\{\sum_{x\in F}e^{\int_0^tf(\phi_s(x)){\rm d}s} : F (t, \varepsilon)\text{ is a finite set which spans } M \right\}.\]
{\it Metric pressure} of a probability measure $\mu$ and a continuous function $f\in C(M, \mathbb{R})$  w.r.t the flow $\phi_t$ can be defined as
\[P_{\mu}(\phi, f)=\lim_{\varepsilon\to 0}\limsup_{t\to\infty}\frac{1}{t}\log N^\mu(\delta, t, \varepsilon, f),\]
where $\delta\in(0,1)$ and
\[N^\mu(\delta, t, \varepsilon, f)=\inf\left\{\sum_{x\in F}e^{\int_0^tf(\phi_s(x)){\rm d}s} \,:\, F \text{ is finite with }\mu\left(\bigcup_{x\in F}B(x, t, \varepsilon)\right)>1-\delta\right\}.\]
He, Yang, and Gao \cite{he2004metrical} proved that if $\mu$ is ergodic and invariant for $\phi$ and $\phi_t$ has no singularity, then
\[P_{\mu}(\phi, f)=h_\mu(\phi_1)+\int f{\rm d}\mu,\]
where $h_\mu(\phi_1)$ is the metric entropy of $\phi_1$ w.r.t $\mu$.

Now let $M$ be a compact boundaryless Riemannian manifold, and let $\mathcal{X}^1(M)$ denote the set of $C^1$ vector fields on $M$. For any $X \in \mathcal{X}^1(M)$, we denote by $\mathrm{Sing}(X) = \{ x \in M : X(x) = 0 \}$ its singular set. Let $\phi_t$ be the flow generated by $X$. The set of all $\phi$ ergodic invariant Borel probability measures on $M$ is denoted by $\mathcal{E}_X(M)$.

To eliminate the influence of singularities, we define pressures using the rescaled Bowen balls introduced in \(\cite{WANG2024128673}\). Let \(I\) be a closed interval containing the origin. If a continuous function \(\alpha: I \to \mathbb{R} \) is strictly monotonically increasing and satisfies \(\alpha(0) = 0\), then \(\alpha\) is called a \textit{reparametrization function}.  Denote by \(\operatorname{Rep}(I)\) the set of all reparametrization functions  on \(I\). Let $x\in M\setminus{\rm Sing}(X)$, define
\[
B_1^*(x,t,\varepsilon,X) = \left\{ y \in M : d(\ps(x), \ps(y)) < \varepsilon \norm{X(\ps(x))} ,\,\forall\, 0 \leq s \leq t \right\}.
\]
\[
B_2^*(x,t,\varepsilon,X) = \left\{ y \in M : \exists \alpha \in \Rep([0,t]), \text{ s.t. } d(\phi_{\alpha(s)}(x), \ps(y)) < \varepsilon \norm{X(\phi_{\alpha(s)}(x))}, \,\forall\, 0 \leq s \leq t \right\}.
\]
\[
B_3^*(x,t,\varepsilon,X) = \left\{ y \in M : \exists \alpha \in \Rep([0,t]), \text{ s.t. } d(\ps(x), \phi_{\alpha(s)}(y)) < \varepsilon \norm{X(\ps(x))} ,\,\forall\, 0 \leq s \leq t \right\}.
\]
Let \( \mu \) be a Borel probability measure on $M$ with \( \mu(\Sing(X)) = 0 \). For \( f \in C(M, \mathbb{R}) \), define
\[
N_i^\mu(\delta, t, \varepsilon, X, f) = \inf_{\substack{F \in \mathcal{F}_i(t, \varepsilon) }} \sum_{x \in F} e^{\int_0^t f(\phi_s(x)) \, {\rm d}s},
\]
where
\[
\mathcal{F}_i(t, \varepsilon) = \left\{ F \subset M \setminus \text{Sing}(X) : F \text{ is finite and }\mu\left( \bigcup_{x \in F} B_i^*(x, t, \varepsilon, X) \right) > 1 - \delta \right\}.
\]

In our paper, we have the following theorems which characterize the rescaled metric pressure for a $C^1$ vector  field with singularities.

\begin{Theorem}\label{mainthm1}
Let \( X \in \mathcal{X}^1(M) \),  \( \phi_t \) be the flow generated by \( X \) and \( \mu \) be a Borel probability measure with \( \mu(\text{Sing}(X)) = 0 \). For any \( 0 < \delta < 1 \) and \( f \in C(M, \mathbb{R}) \), we have
\[
\lim_{\varepsilon \to 0} \limsup_{t \to \infty} \frac{1}{t} \log N_1^\mu(\delta, t, \varepsilon, X, f) = \lim_{\varepsilon \to 0} \limsup_{t \to \infty} \frac{1}{t} \log N_2^\mu(\delta, t, \varepsilon, X, f) = \lim_{\varepsilon \to 0} \limsup_{t \to \infty} \frac{1}{t} \log N_3^\mu(\delta, t, \varepsilon, X, f),
\]
\[
\lim_{\varepsilon \to 0} \liminf_{t \to \infty} \frac{1}{t} \log N_1^\mu(\delta, t, \varepsilon, X, f) = \lim_{\varepsilon \to 0} \liminf_{t \to \infty} \frac{1}{t} \log N_2^\mu(\delta, t, \varepsilon, X, f) = \lim_{\varepsilon \to 0} \liminf_{t \to \infty} \frac{1}{t} \log N_3^\mu(\delta, t, \varepsilon, X, f).
\]
\end{Theorem}

In \cite{he2004metrical} and \cite{WangChen}, Katok's formula of metric pressure for nonsingular flow has been considered.  The following theorem establishes that for a $C^1$ vector field $X$, if  $\log\|X\|$ is integrable, then Katok's formula for metric pressure defined through rescaled Bowen balls $B_i^*(x, t, \varepsilon, X)$ also holds.

\begin{Theorem}\label{mainthm2}
If \( \mu \) is a $\phi$ ergodic invariant measure with \( \mu(\Sing(X)) = 0 \) and \( \int \log \norm{X(x)} \dmu < \infty \), then we have
\[
\lim_{\varepsilon \to 0} \limsup_{t \to \infty} \frac{1}{t} \log N_i^\mu(\delta, t, \varepsilon, X, f) =\lim_{\varepsilon \to 0} \liminf_{t \to \infty} \frac{1}{t} \log N_i^\mu(\delta, t, \varepsilon, X, f) = h_\mu(\phi_1) + \int f \, {\rm d}\mu
\]
for $i=1,2,3$.
\end{Theorem}

By the above theorems, the {\it rescaled metric pressure } $P_\mu^*(X, f)$ for the $C^1$ vector field $X$ and the Borel probability measure $\mu$ and $f\in C(M, \mathbb{R})$ is naturally defined by the following:
\[P_\mu^*(X, f)=\lim_{\delta\to 0}\lim_{\varepsilon \to 0} \limsup_{t \to \infty} \frac{1}{t} \log N_1^\mu(\delta, t, \varepsilon, X, f).\]
Theorem \ref{mainthm2} says that if $\mu$ is a $\phi$ ergodic invariant measure with $\mu({\rm Sing}(X))=0$ and $\log\|X(x)\|$ is integrable, then $P_\mu^*(X, f)=P_\mu(\phi, f)$, and they are all equal to the metric pressure of $\mu, f$ with respect to the time 1 map $\phi_1: M\to M$.

Now we consider the rescaled topological pressure, which is also defined using the rescaled Bowen balls. Let $K\subset M\setminus {\rm Sing}(X)$ be a compact set. Given $t>0$ and $\varepsilon>0$, we say that $F$ is a {\it rescaled i-$(t, \varepsilon, K)$-spanning set} if $F\subset K$ and
$$K\subset\bigcup_{x\in F}B^*_i(x, t, \varepsilon, X).$$
For any given $f\in C(M, \mathbb{R})$ and $i=1,2,3$, let
\[
N_{i,t}^*(X, f, \varepsilon, K) = \inf\left\{ \sum_{x \in F} e^{\int_0^t f(\phi_s x) {\rm d}s} : F \text{ is a finite rescaled i-} (t,\varepsilon,K)\text{-spanning set} \right\},
\]
\[
{P}_i^*(X, f, \varepsilon,K) = \limsup_{t \to \infty} \frac{1}{t} \log N_{i,t}^*(X, f, \varepsilon,K),
\]
and then
\[
{P}_i^*(X, f) = \sup\limits_{K} \lim_{\varepsilon \to 0} {P}_i^*(X, f, \varepsilon,K),
\]
where the supremum is over the compact subsets \( K \subset M\setminus {\rm Sing}(X) \). If $X$ is nonsingular, then by the fact that $\|X(x)\|$ has an upper bound and a positive lower bound, we can see that $P_1^*(X, f)=P(\phi, f)$ for every $f\in C(M, \mathbb{R})$.

Inspired by Franco \(\cite{franco1977flows}\),  who also defined the topological pressure for flows using separating set, we can define the rescaled topological pressure by the notion of rescaled separating sets. Given a compact set \( K \subset M\setminus{\rm Sing}(X) \), \( t>0 \) and \( \varepsilon > 0 \), we say that \( E \) is a {\it rescaled i-\((t, \varepsilon, K) \)-separating set } if \( E \subset K \) and
\[ B_i^*(x, t, \varepsilon, X) \cap E = \{ x \} \]
for all \( x \in E \).  For any given $f\in C(M, \mathbb{R})$, let
\[
Z_{i,t}^*(X, f, \varepsilon,K) = \sup\left\{ \sum_{x \in E} e^{ \int_0^t f(\phi_s x) {\rm d}s} : E \text{ is a rescaled i-}(t,\varepsilon,K)\text{-separating set} \right\},\quad i=1,2,3.
\]
\[
{Q}_i^*(X, f, \varepsilon,K) = \limsup_{t \to \infty} \frac{1}{t} \log Z_{i,t}^*(X, f, \varepsilon,K),\quad i=1,2,3,
\]
and then
\[
{Q}_i^*(X, f) = \sup\limits_{K} \lim_{\varepsilon \to 0} {Q}_i^*(X, f, \varepsilon,K),\quad i=1,2,3,
\]
where the supremum is over the compact subsets \( K \subset M\setminus {\rm Sing}(X) \).

We have the following theorem.

\begin{Theorem}\label{theorem 1.3}
Let \( M \) be a compact boundaryless Riemannian manifold, $X\in\mathcal{X}^1(M)$ and $f\in C(M, \mathbb{R})$, we have
\[
{P}_1^*(X, f) = {P}_2^*(X, f)={P}_3^*(X,  f)={Q}_1^*(X, f) = {Q}_2^*(X, f)={Q}_3^*(X,  f).
\]
\end{Theorem}

Thus we can define the {\it rescaled topological pressure} for the $C^1$ vector field $X$ and the continuous function $f\in C(M, \mathbb{R})$ as following
$$P^*(X, f)={P}_1^*(X, f).$$

For the rescaled topological pressure, we have the following partial variational principle.

\begin{Theorem}\label{theorem1.4:vp}
For every \( C^1 \) vector field \( X \) of a compact boundaryless Riemannian manifold $M$, and any continuous function \( f \in C(X, \mathbb{R}) \), one has
\[
\sup\{ {P}_\mu^*(X,f) : \mu \text{ is a Borel probability measure with } \mu(\text{Sing}(X)) = 0 \} \leq {P}^*(X,f).
\]
\end{Theorem}
The remainder of the paper is organized as follows:
In Section \ref{sec:boundedvariation}, we prove the bounded variation property of $f\in C(M, \mathbb{R})$  for the rescaled balls $B_i^*(x, t, \varepsilon, X)$. From the results of this section it follows that the bounded variation hypothesis in Theorem 1.2 of Wang-Chen \(\cite{WangChen}\) is automatically satisfied in the case of singular flows generated by \(C^1\) vector fields. Actually, we can remove such a bounded variation assumption even in the case of nonsingular flows on compact metric spaces (see Appendix A). In Section \ref{sec:pressure}, we prove Theorem \ref{mainthm1}. In Section \ref{sec:Katokformula}, we establish Katok's formula for the rescaled metric pressure, which reveal the connection between the rescaled metric pressure and the entropy of the flow with singularities. In Section \ref{sec:tp}, we give the proof of Theorem \ref{theorem 1.3} and \ref{theorem1.4:vp}. And we prove a result similar to Theorem \ref{mainthm2} in Appendix which improves the main results in \cite{WangChen}.

\section{Bounded Variation}\label{sec:boundedvariation}

For a continuous function \( f \in C(M, \mathbb{R}) \), \(t>0,\,\varepsilon>0\), define

\[
\gamma_{i, t}(f, \varepsilon) = \sup\limits_{x \in M \setminus \text{Sing}(X)}  \left\{ \left| \int_0^t  \left(f(\phi_s(y)) - f(\phi_s(z))\right) {\rm d}s \,\right| : y,z \in B_i^*(x, t,\varepsilon, X) \right\},\quad i=2,3.
\]
In this section, we will prove the following proposition, and the property what we have in the proposition is called bounded variation in the paper \cite{WangChen}. It will be used in the proof of Theorem \ref{mainthm1} and \ref{theorem 1.3}.

\begin{Proposition} \label{boundedvariation}
For any continuous function $f\in C(M, \mathbb{R})$, we have
\[ \lim\limits_{\varepsilon \to 0} \lim\limits_{t \to \infty} \frac{\gamma_{i, t}(f, \varepsilon) }{t} = 0, \quad i = 2, 3. \]

\end{Proposition}

To prove the above proposition, we prepare a relative uniform flowbox theorem which was firstly proved by Wen-Wen \cite{wen2019rescaled}. For any \(x\in M\setminus \Sing(X)\) and \(r>0,\) set
\[
N_x = \{ v \in T_x M : v \perp X(x) \},\ \ \ \
U_x(r \| X(x) \|) = \left\{ v + t X(x) : v \in N_x,\| v \| \leq r \| X(x) \|, |t|\leq r  \right\}.
\]
The following lemma is Proposition 2.2 of \cite{wen2019rescaled}.
\begin{Lemma}\label{flowbox}
For a \( C^1 \) vector field \( X \) on \( M \), there exists \( r_0 > 0 \) such that for any \( x \in M \setminus \text{Sing}(X) \), there exists a \( C^1 \) embedding
\[
F_x: U_x(r_0\| X(x) \|) \to M ,\,F_x(v + t X(x)) = \phi_t\bigl(\exp_x(v)\bigr),
\]
which satisfies
 \(\| D_pF_x \| \leq 3, m(D_pF_x) \geq \frac{1}{3}, \forall\, p \in U_x(r_0\| X(x) \|), x\in M\setminus{\rm Sing}(X)\).
\end{Lemma}

By the compactness of $M$,  for any flow $\phi_t$ generated by a \(C^1\)  vector field $X$ on  $M$, there always exists a (Lipschitz) constant $L > 0$ such that for any  \( t \in \mathbb{R} \) and \( x, y \in M \), one has
\[
d\bigl(\phi_t(x), \phi_t(y)\bigr) \leq e^{L |t|} d(x, y).
\]
Here we point out that the constant \( r_0 \) in Lemma \ref{flowbox} only depends on the constant \( L \).

The following lemma is Lemma 4 of \cite{WANG2024128673}.
\begin{Lemma}\label{controloftimereparametrization}
Let \( X \) be a \( C^1 \) vector field on \( M \), and \( \phi_t \) be the flow generated by \( X \). For any \( \lambda > 0 \) and \( 0 < b < r_0 \), there exists \( \varepsilon > 0 \) such that for any \( x, y \in M \setminus \text{Sing}(X) \) and \( T \in [b, +\infty) \), if
\[
d\bigl(\phi_t(x), \phi_{\alpha(t)}(y)\bigr) \leq \varepsilon \| X\bigl(\phi_t(x)\bigr) \|,\quad \forall\, t \in [0, T]
\]
or
\[
d\bigl(\phi_{\alpha(t)}(x), \phi_t(y)\bigr) \leq \varepsilon \| X\bigl(\phi_{\alpha(t)}(x)\bigr) \|,\quad \forall\, t \in [0, T],
\]
then
\[
|\alpha(t) - t| \leq \begin{cases}
\lambda t \, , & b \leq t \leq T; \\
\lambda b,  & 0 \leq t \leq b.
\end{cases}
\]
\end{Lemma}
\noindent

We comment here that although the part \( |\alpha(t) - t| \leq \lambda b \) , \( 0 \leq t \leq b \) in the conclusion is not contained in the original statement of Lemma 4 in \cite{WANG2024128673}, but we can conclude it directly from the proof there.

No we can give the proof of Proposition \ref{boundedvariation}.

\bigskip

\noindent{\bf Proof of Proposition \ref{boundedvariation}.}
Let an arbitrary $f\in C(M, \mathbb{R})$ be given. To prove the result of Proposition \ref{boundedvariation}, it suffices to prove that for any \(\sigma\in(0, 1)\), there exist \(\varepsilon_0 > 0\) and \(T > 0\) such that
for any \(x \in M \setminus \text{Sing}(X)\), \(0 < \varepsilon \leq \varepsilon_0\), \(t > T\) and any \(y,z \in B_{i}^*(x,t,\varepsilon,X)\), we have
\[
\frac{1}{t} \left| \int_0^t ( f(\phi_s(y)) - f(\phi_s(z))){\rm d}s \right| < \sigma.
\]

Let \(\| f \| = \sup\limits_{x \in M} |f(x)|\) be the $C^0$ norm of $f$. For any given \(\sigma > 0\), by the uniform continuity of \(f\), we can choose \(\eta > 0\) such that for any \(x,y \in M\), if \(d(y,z) < \eta  \sup\limits_{x \in M} \| X(x) \|\), then
\[
|f(y) - f(z)| < \frac{\sigma}{8(\| f \| + 1)}.
\]

By Lemma \ref{flowbox}, we can choose \(0 < r < {\eta}/{12}\), such that for any \(x \in M \setminus \text{Sing}(X)\) and \(y,z \in \phi_{[-2r,2r]}(x)\), one has \(d(y,z) < \eta \| X(x) \|\). Furthermore, by Lemma \ref{controloftimereparametrization}, we can choose \(0 < \varepsilon_0 < \eta\) such that for any \(x,y \in M \setminus \text{Sing}(X)\) and \(\alpha \in \text{Rep}[0,t]\) with \(t \geq r\), if
\[
d\bigl( \phi_{\alpha(s)}(x), \phi_s(y) \bigr) < \varepsilon_0 \| X(\phi_{\alpha(s)}(x)) \|, \quad \forall s \in [0,t]
\]
or
\[
d\bigl(\phi_{s}(x), \phi_{\alpha(s)}(y)\bigr) < \varepsilon_0 \| X\bigl(\phi_{s}(x)\bigr) \|,\quad \forall\, s \in [0, t],
\]
then
\[
\left|\alpha(s) - s\right| <
\begin{cases}
\dfrac{\sigma}{8(\|f\| + 1)}  s, & r \leq s \leq t; \\[10pt]
r, &  0 \leq s \leq r.
\end{cases}
\]

Now we assume that \(0 < \varepsilon \leq \varepsilon_0\), \(t \geq r\), \(x \in M \setminus \text{Sing}(X)\) and \(y,z \in B_{2}^*(x,t,\varepsilon,X)\). Denote by \(n = \left\lfloor \frac{t}{r} \right\rfloor\).
One can easily check that
\[
\left| \int_0^t f(\phi_s(y)) \, {\rm d}s - \int_0^t f(\phi_s(z)) \, {\rm d}s - \left( \int_0^{nr} f(\phi_s(y)) \, {\rm d}s - \int_0^{nr} f(\phi_s(z)) \, {\rm d}s \right) \right| \\
\leq 2 \|f \| (t - nr) < 2 \| f\| r.
\]

In the following, we have an estimation on
\[
\left| \int_0^{nr} f(\phi_s(y)) \, {\rm d}s - \int_0^{nr} f(\phi_s(z)) \, {\rm d}s \right|.
\]
\noindent
Since \( y \in B_2^*(x,t,\varepsilon,X) \), there exists \( \alpha \in \text{Rep}[0,t] \) such that
\[
d\bigl( \phi_{\alpha(s)}(x), \phi_s(y) \bigr) < \varepsilon \| X(\phi_{\alpha(s)}(x)) \|, \quad \forall s\in [0,t].
\]
\noindent
Since \( \varepsilon \leq \varepsilon_0 < \eta \) and by the choice of \( \eta \), we have

\[
\left| f(\phi_{\alpha(s)}(x)) - f(\phi_s(y)) \right| < \frac{\sigma}{8(\| f \| + 1)}, \quad \forall s \in [0,t].
\]
Thus
\[
\left| \int_0^{nr} f(\phi_{\alpha(s)}(x)) \, {\rm d}s - \int_0^{nr} f(\phi_s(y)) \, {\rm d}s \right| < \frac{\sigma}{8(\| f \| + 1)} \cdot nr \leq \frac{\sigma}{8(\| f \| + 1)} t.
\]
Note that
\begin{align*}
&\left| \int_0^{nr}f(\phi_{\alpha(s)}(x)) \, {\rm d}s - \int_0^{nr} f(\phi_s(x)) \, {\rm d}s \right| \\
\leq &\left| \int_0^{nr} f(\phi_{\alpha(s)}(x)) \, {\rm d}s - \int_0^{\alpha(nr)} f(\phi_s(x)) \, {\rm d}s \right| + \left| \int_0^{\alpha(nr)} f(\phi_s(x)) \, {\rm d}s - \int_0^{nr} f(\phi_s(x)) \, {\rm d}s \right|.
\end{align*}
It is easy to see that
\[
\left| \int_0^{\alpha(nr)} f(\phi_s(x)) \, {\rm d}s - \int_0^{nr} f(\phi_s(x)) \, {\rm d}s \right| \leq \| f \| \cdot |\alpha(nr) - nr| < \| f\| \cdot \frac{\sigma}{8(\| f\| + 1)} nr < \frac{\sigma}{8} t.
\]
On the other hand, by the Mean Value Theorem for Integrals, we have
\begin{align*}
&\left| \int_0^{nr} f(\phi_{\alpha(s)}(x)) \, {\rm d}s - \int_0^{\alpha(nr)} f(\phi_s(x)) \, {\rm d}s \right| \\
=& \left| \displaystyle\sum_{k=1}^n \bigl(\int_{(k-1)r}^{kr} f(\phi_{\alpha(s)}(x)) \, {\rm d}s - \int_{\alpha((k-1)r)}^{\alpha(kr)} f(\phi_s(x)) \, {\rm d}s \bigr) \right| \\
=& \left| \sum_{k=1}^n \left( f(\phi_{\alpha(s_k)}(x)) \cdot r - f(\phi_{\alpha(s_k')}(x)) \cdot \bigl( \alpha(kr) - \alpha((k-1)r) \bigr) \right) \right|,
\end{align*}
where \(s_k, s_k' \in [(k-1)r, kr].\)

\bigskip

\noindent
$\textbf{Claim: }
\emph{For any } 1 \leq k \leq n, \emph{ we have }
|\alpha(s_k) - \alpha(s_k')| < 4r,\, |\alpha(kr) - \alpha((k-1)r) - r| < \frac{\sigma}{8(\| f \| + 1)} r.$

\noindent
\begin{proof} Given $ 1 \leq k \leq n$, let $\bar{x} = \phi_{\alpha((k-1)r)}(x)$, $\bar{y} = \phi_{(k-1)r}(y)$ and $\bar{\alpha}(t) = \alpha(t + (k-1)r) - \alpha((k-1)r)$. Then $\bar{\alpha} \in \text{Rep}[0, t - (k-1)r]$  and for any $0 \leq s \leq t - (k-1)r$, we have
\[
d\bigl( \phi_{\bar{\alpha}(s)}(\bar{x}), \phi_s(\bar{y}) \bigr) = d\bigl( \phi_{\alpha(s + (k-1)r)}(x), \phi_{s + (k-1)r}(y) \bigr) < \varepsilon \| X(\phi_{\alpha(s + (k-1)r)}(x)) \|.
\]
\noindent
Since  $0 < \varepsilon \leq \varepsilon_0$, by Lemma \ref{controloftimereparametrization} we have
\begin{align*}
|\alpha(s_k) - \alpha(s_k')| &= |\bar{\alpha}(s_k - (k-1)r) - \bar{\alpha}(s_k' - (k-1)r)| \\
&\leq \left| \bar{\alpha}(s_k - (k-1)r) - (s_k - (k-1)r) \right| + |s_k - s_k'| + \left| \bar{\alpha}(s_k' - (k-1)r) - (s_k' - (k-1)r) \right| \\
&< r + r + r < 4r,
\end{align*}
and
\[
\left| \alpha(kr) - \alpha((k-1)r) - r \right| = \left| \bar{\alpha}(r) - r \right| < \frac{\sigma}{8(\| f \| + 1)} r.
\]
This finishes the proof of the claim.
\hfill \qedsymbol

\bigskip

By the fact that $|\alpha(s_k)-\alpha(s_k')|< 4r$, there is $t_0\in\mathbb{R}$ such that $\phi_{\alpha(s_k)}(x), \phi_{\alpha(s_k')}(x)\in \phi_{[-2r, 2r]}(\phi_{t_0}(x))$, hence by the choice of $r$ we have $d(\phi_{\alpha(s_k)},\phi_{\alpha(s_k')})<\eta\|X(\phi_{t_0}(x))\|$, and then by the choice of $\eta$, we have
\[
\left|  f(\phi_{\alpha(s_k)}(x)) - f(\phi_{\alpha(s'_{k})}(x))  \right| < \frac{\sigma}{8(\|f||+1)}.
\]
Thus
\begin{align*}
&\left| \int_0^{nr} f\bigl( \phi_{\alpha(s)}(x) \bigr) {\rm d}s - \int_0^{\alpha(nr)} f\bigl( \phi_s(x) \bigr) {\rm d}s \right| \\
= \displaystyle &\sum_{k=1}^n \left| f\bigl( \phi_{\alpha(s_k)}(x) \bigr) \cdot r -f\bigl( \phi_{\alpha(s_k')}(x) \bigr) \bigl( \alpha(kr) - \alpha\bigl((k - 1)r\bigr) \bigr) \right| \\
\leq &\displaystyle\sum_{k=1}^n \left| f\bigl( \phi_{\alpha(s_k)}(x) \bigr) - f\bigl( \phi_{\alpha(s_k')}(x) \bigr) \right| \cdot r + \displaystyle\sum_{k=1}^n \left| f\bigl( \phi_{\alpha(s_k')}(x) \bigr) \right| \cdot \left| \alpha(kr) - \alpha\bigl((k - 1)r\bigr) - r \right| \\
<& \frac{\sigma}{8(\| f \| + 1)}  nr + \| f\|  \frac{\sigma}{8(\| f \| + 1)}  nr=\frac{\sigma}{8} nr \leq \frac{\sigma}{8} t.
\end{align*}

\noindent
Then we have
\begin{align*}
&\left| \int_0^{nr} f(\phi_{\alpha(s)}(x)) \, {\rm d}s - \int_0^{nr} f(\phi_s(x)) \, {\rm d}s \right| \\
\leq &\left| \int_0^{nr} f(\phi_{\alpha(s)}(x)) \, {\rm d}s - \int_0^{\alpha(nr)} f(\phi_s(x)) \, {\rm d}s \right| + \left| \int_0^{\alpha(nr)} f(\phi_s(x)) \, {\rm d}s - \int_0^{nr} f(\phi_s(x)) \, {\rm d}s \right| \\
< &\frac{\sigma}{8} t + \frac{\sigma}{8} t = \frac{\sigma}{4} t.
\end{align*}
Thus
\begin{align*}
&\left| \int_0^{nr} f(\phi_s(y)) \, {\rm d}s - \int_0^{nr} f(\phi_s(x)) \, dx \right| \\
\leq &\left| \int_0^{nr} f(\phi_s(y)) \, {\rm d}s - \int_0^{nr} f(\phi_{\alpha(s)}(x)) \, {\rm d}s \right| + \left| \int_0^{nr} f(\phi_{\alpha(s)}(x)) \, {\rm d}s - \int_0^{nr} f(\phi_s(x)) \, {\rm d}s \right| \\
< &\frac{\sigma}{8(\| f \| + 1)} t + \frac{\sigma}{4} t < \frac{3\sigma}{8} t.
\end{align*}
Similarly, we can prove that
\[
\left| \int_0^{nr} f(\phi_s(z)) \, {\rm d}s - \int_0^{nr} f(\phi_s(x)) \, {\rm d}s \right| < \frac{3\sigma}{8} t.
\]
In conclusion, we get that
\begin{align*}
&\left| \int_0^{nr} f(\phi_s(y)) \, {\rm d}s - \int_0^{nr} f(\phi_s(z)) \, {\rm d}s \right| \\
\leq &\left| \int_0^{nr} f(\phi_s(y)) \, {\rm d}s - \int_0^{nr} f(\phi_s(x)) \, {\rm d}s \right| + \left| \int_0^{nr} f(\phi_s(z)) \, {\rm d}s - \int_0^{nr} f(\phi_s(x)) \, {\rm d}s \right| \\
<& \frac{3\sigma}{8} t + \frac{3\sigma}{8} t = \frac{3\sigma}{4} t.
\end{align*}

\noindent
Thus
\[
\left| \int_0^t f(\phi_s(y)) \, {\rm d}s - \int_0^t f(\phi_s(z)) \, {\rm d}s \right|
< \frac{3\sigma}{4} t + 2 \| f \| r.
\]

$\text{Choose } T > r \text{ such that } \frac{\sigma}{4} \cdot T \geq 2 \| f \| r. \text{ Then for any } t > T, \text{one has}$
\[
\left| \int_0^t f(\phi_s(y)) \, {\rm d}s - \int_0^t f(\phi_s(z)) \, {\rm d}s \right| < \sigma t,
\]
or
\[
\frac{1}{t} \left| \int_0^t f(\phi_s(y)) \, {\rm d}s - \int_0^t f(\phi_s(z)) \, {\rm d}s \right| < \sigma.
\]
This proves that for any \(0 < \varepsilon \leq \varepsilon_0\), \(t > T\) and for any \(x \in M \setminus \text{Sing}(X)\) , \(y,z \in B_{2}^*(x,t,\varepsilon,X)\),
\[
\frac{1}{t} \left| \int_0^t \left( f(\phi_s(y)) - f(\phi_s(z)) \right) {\rm d}s \right| < \sigma,
\]

\noindent
i.e.
\[
\lim\limits_{\varepsilon \to 0} \lim_{t \to \infty} \frac{\gamma_{2,t}(f,\varepsilon)}{t} = 0.
\]

Now we're going to prove that
\[
\lim\limits_{\varepsilon \to 0} \lim\limits_{t \to \infty} \frac{\gamma_{3,t}f, \varepsilon)}{t} = 0.
\]
We still choose the previous constants \( \varepsilon_0 \) and \( r \).
Then for any \( t > r \),  \( x \in M \setminus \text{Sing}(X) \) and
\( y, z \in B_3^*(x, t, \varepsilon, X) \), we still take a similar estimation on
\[
\left| \int_0^t f(\phi_s(y)) \, {\rm d}s - \int_0^t f(\phi_s(z)) \, {\rm d}s \right|.
\]
Denote \( n = \left\lfloor \frac{t}{r} \right\rfloor \). Then we still have
\[
\left| \int_0^t f(\phi_s(y)) \, {\rm d}s - \int_0^t f(\phi_s(z)) \, {\rm d}s - \left( \int_0^{nr} f(\phi_s(y)) \, {\rm d}s - \int_0^{nr} f(\phi_s(z)) \, {\rm d}s \right) \right| \leq 2 \| f \| (t - nr) \leq 2 \| f \| r.
\]
Thus
\[
\left| \int_0^t f(\phi_s(y)) \, {\rm d}s - \int_0^t f(\phi_s(z)) \, {\rm d}s \right| \\
\leq \left|   \int_0^{nr} \left(f(\phi_s(y))-f(\phi_s(z))\right) \, {\rm d}s \ \right| + 2 \| f \|  r.
\]
Similar to the discussion of \( \gamma_{2,t}(f, \varepsilon) \), we have
\[\left| \int_0^{nr} \left(f(\phi_s(y)) -f(\phi_s(z)) \right)\, {\rm d}s \right|\leq \left| \int_0^{nr}\left( f(\phi_s(y)) -f(\phi_s(x)) \right)\, {\rm d}s \right| + \left| \int_0^{nr}\left( f(\phi_s(x))-f(\phi_s(z))\right) \, {\rm d}s \right|.
\]
Since $y\in B_3^*(x, t, \varepsilon, X)$, there is $\beta\in {\rm Rep}[0, t]$ such that $d(\phi_s(x), \phi_{\beta(s)}(y))<\varepsilon\|X(\phi_s(x))\|,\, \forall s\in[0,t]$.  We have
\[
\left| \int_0^{nr} \left(f(\phi_s(y))-f(\phi_s(x))\right) \, {\rm d}s \right| \leq \left| \int_0^{nr} \left(f(\phi_s(y)) -f(\phi_{\beta(s)}(y))\right) \, {\rm d}s \right| + \left| \int_0^{nr} \left(f(\phi_{\beta(s)}(y))-f(\phi_s(x))\right) \, {\rm d}s \right|.
\]
Since $d(\phi_{\beta(s)}(y), \phi_s(x))<\varepsilon_0\|X(\phi_s(x))\|<\eta\|X(\phi_s(x))\|$, we have
\[
\left| \int_0^{nr} \left(f(\phi_{\beta(s)}(y))-f(\phi_s(x))\right) \, {\rm d}s \right| < \frac{\sigma}{8(\| f \| + 1)} nr \leq \frac{\sigma}{8(\| f \| + 1)}  t,
\]
by the choice of $\eta$. On the other hand, we have
\begin{align*}
    &\left| \int_0^{nr} \left(f(\phi_s(y))-f(\phi_{\beta(s)}(y))\right) \, {\rm d}s \right| \\
\leq &\left| \int_0^{\beta(nr)} f(\phi_s(y)) \, {\rm d}s - \int_0^{nr} f(\phi_{\beta(s)}(y)) \, {\rm d}s \right| + \| f \|  |nr - \beta(nr)| \\
< &\left| \int_0^{\beta(nr)} f(\phi_s(y)) \, {\rm d}s - \int_0^{nr} f(\phi_{\beta(s)}(y)) \, {\rm d}s \right| + \| f \|  \frac{\sigma}{8(\| f \| + 1)} nr.
\end{align*}
Similarly, we can choose \( s_k, s_k' \in [(k-1)r, kr] \), \(\forall \,k = 1, \dots, n \) such that
\begin{align*}
  &\left| \int_0^{\beta(nr)} f(\phi_s(y)) \, {\rm d}s - \int_0^{nr} f(\phi_{\beta(s)}(y)) \, {\rm d}s \right| \\
= &\left| \sum_{k=1}^n  f(\phi_{\beta(s'_{k})}(y))\cdot (\beta(kr)-\beta((k-1)r))-\sum_{k=1}^n f(\phi_{\beta(s_k)}(y))\cdot r \right|,
\end{align*}
and prove that
\[
|\beta{(s_k)} - \beta{(s_k')}| < 4r, \quad \left| \beta(kr) - \beta((k-1)r) - r \right| < \frac{\sigma}{8(\| f \| + 1)} r.
\]

\noindent
Thus
\[
\left| \int_0^{\beta(nr)} f(\phi_s(y)) \, {\rm d}s - \int_0^{nr} f(\phi_{\beta(s)}(y)) \, {\rm d}s \right| < \frac{\sigma}{8} t.
\]

\noindent
\text{Then}
\[
\left| \int_0^{nr} f(\phi_s(y)) \, {\rm d}s - \int_0^{nr} f(\phi_{\beta(s)}(y)) \, {\rm d}s \right| <\|f\||\beta(nr)-nr|+ \frac{\sigma}{8} t<\frac{\sigma}{4}t.
\]
Consequently we have
\[
\left| \int_0^{nr} f(\phi_s(y)) \, {\rm d}s - \int_0^{nr} f(\phi_s(x)) \, {\rm d}s \right| < \frac{\sigma}{8(\| f \| + 1)} t + \frac{\sigma}{4} t < \frac{3\sigma}{8} t.
\]
Similarly, we have
\[
\left| \int_0^{nr} f(\phi_s(z)) \, {\rm d}s - \int_0^{nr} f(\phi_s(x)) \, {\rm d}s \right| < \frac{3}{8} \sigma t,
\]
and then
\[
\left| \int_0^{t} f(\phi_s(y)) \, {\rm d}s - \int_0^t f(\phi_s(z)) \, {\rm d}s \right| < \frac{3}{4} \sigma t + 2 \| f \| r.
\]
Similarly, for any $t > T$  where $T$ has been chosen as above, we have
\[\left| \int_0^t f(\phi_s(y)) \, {\rm d}s - \int_0^t f(\phi_s(z)) \, {\rm d}s \right| < \sigma t. \]
This proves that
\[
\lim\limits_{\varepsilon \to 0} \lim_{t \to \infty} \frac{\gamma_{3,t}(f,\varepsilon)}{t} = 0,
\]
and finishes the proof of Proposition \ref{boundedvariation}.
\end{proof}

\section{Rescaled Metric Pressure}\label{sec:pressure}
In this section, we prove Theorem \ref{mainthm1}, which guarantees that regardless of which Bowen ball $B_i^*(x, t, \varepsilon, X)$ we use, the resulting metric pressures are identical.

Here we recall an existing result to prove Theorem \ref{mainthm1}. The following lemma, derived from Lemma \ref{controloftimereparametrization}, is Lemma 5 in \cite{WANG2024128673}.

\begin{Lemma}\label{lemma3.1}
Let \( X \) be a \( C^1 \) vector field on \( M \), and let \( \phi_t(x) \) be the flow generated by \( X \). For any given \( \lambda > 0 \) and \( b \in (0, r_0) \), there always exists \( \varepsilon_0 > 0 \) such that for any \( x \in M \setminus \text{\rm Sing}(X) \), \( \varepsilon \in (0, \varepsilon_0] \) and any $t>b$, we have
\[
B_3^*(x, t, \varepsilon, X) \subseteq B_2^*(x, (1 - \lambda)t, \varepsilon, X),
\]
\[
B_2^*(x, t, \varepsilon, X) \subseteq B_3^*(x, (1 - \lambda)t, \varepsilon, X).
\]
\end{Lemma}

The following lemma is a part of Theorem \ref{mainthm1}.

\begin{Lemma}\label{lemma3.2}
Let $\mu$ be a Borel probability measure with $\mu(\text{Sing}(X)) = 0$ and $\delta \in (0, 1)$, $f \in C(M, \mathbb{R})$ be given. We have
\[
\lim_{\varepsilon \to 0} \limsup_{t \to \infty} \frac{1}{t} \log N_2^\mu(\delta, t, \varepsilon, X, f) = \lim_{\varepsilon \to 0} \limsup_{t \to \infty} \frac{1}{t} \log N_3^\mu(\delta, t, \varepsilon, X, f),
\]
\[
\lim_{\varepsilon \to 0} \liminf_{t \to \infty} \frac{1}{t} \log N_2^\mu(\delta, t, \varepsilon, X, f) = \lim_{\varepsilon \to 0} \liminf_{t \to \infty} \frac{1}{t} \log N_3^\mu(\delta, t, \varepsilon, X, f).
\]
\end{Lemma}

\begin{proof}
Recall that \[
N_i^\mu(\delta, t, \varepsilon, X, f) = \inf_{\substack{F \in \mathcal{F}_i(t, \varepsilon) }} \sum_{x \in F} e^{\int_0^t f(\phi_s(x)) \, {\rm d}s},
\]
where
\[
\mathcal{F}_i(t, \varepsilon) = \left\{ F \subset M \setminus \text{Sing}(X) \,\big|\, F \text{ is finite and }\mu\left( \bigcup_{x \in F} B_i^*(x, t, \varepsilon, X) \right) > 1 - \delta \right\}.
\]
By Lemma \ref{lemma3.1}, we can see that for any given $\lambda > 0$, if $\varepsilon>0$ is chosen sufficiently small and $t$ is chosen sufficiently large, then we have $\mathcal{F}_2(t, \varepsilon)\subset \mathcal{F}_3((1 - \lambda)t, \varepsilon)$. Thus

\begin{align*}
    N_2^\mu(\delta, t, \varepsilon, X, f) &= \inf_{F \in \mathcal{F}_2(t, \varepsilon)} \sum_{x \in F} e^{\int_0^t f(\phi_s(x)) \, {\rm d}s}\geq \inf_{F \in \mathcal{F}_3((1 - \lambda)t, \varepsilon)} \sum_{x \in F} e^{\int_0^{t} f(\phi_s(x)) \, {\rm d}s} \\
    &\geq e^{-\lambda \|f\| t} \inf_{F \in \mathcal{F}_3((1 - \lambda)t, \varepsilon)}\sum_{x \in F} e^{\int_0^{(1 - \lambda)t} f(\phi_s(x)) \, {\rm d}s}\\
    &=e^{-\lambda \|f\| t}N_3^\mu(\delta, (1 - \lambda)t, \varepsilon, X, f).
\end{align*}

\noindent
Thus
\[
\frac{1}{t} \log N_2^\mu(\delta, t, \varepsilon, X, f) \geq -\lambda \|f\| + \frac{1}{t} \log N_3^\mu(\delta, (1 - \lambda)t, \varepsilon, X, f).
\]

\noindent
Then once $\varepsilon > 0$ is chosen sufficiently small, we have
\begin{align*}
\limsup_{t \to \infty} \frac{1}{t} \log N_2^\mu(\delta, t, \varepsilon, X, f)&
\geq -\lambda \|f\| +(1-\lambda) \limsup_{t \to \infty} \frac{1}{(1 - \lambda)t} \log N_3^\mu(\delta, (1 - \lambda)t, \varepsilon, X, f)\\
&=-\lambda \|f\| +(1-\lambda)\limsup_{t \to \infty} \frac{1}{t} \log N_3^\mu(\delta, t, \varepsilon, X, f).
\end{align*}

\noindent
Letting $\varepsilon \to 0$, we obtain
\[
\lim_{\varepsilon \to 0} \limsup_{t \to \infty} \frac{1}{t} \log N_2^\mu(\delta, t, \varepsilon, X, f) \\
\geq-\lambda \|f\|+
(1-\lambda)\lim_{\varepsilon \to 0} \limsup_{t \to \infty} \frac{1}{t} \log N_3^\mu(\delta,t, \varepsilon, X, f).
\]

\noindent
 Letting $\lambda \to 0$, we have
\[
\lim_{\varepsilon \to 0} \limsup_{t \to \infty} \frac{1}{t} \log N_2^\mu(\delta, t, \varepsilon, X, f) \geq \lim_{\varepsilon \to 0} \limsup_{t \to \infty} \frac{1}{t} \log N_3^\mu(\delta, t, \varepsilon, X, f).
\]

Similarly, from Lemma \ref{lemma3.1} we can also get that for any given $\lambda > 0$, one has $\mathcal{F}_3(t, \varepsilon)\subset \mathcal{F}_2((1 - \lambda)t, \varepsilon)$ for sufficiently small $\varepsilon>0$ and sufficiently large $t$, then we can prove that
\[
\lim_{\varepsilon \to 0} \limsup_{t \to \infty} \frac{1}{t} \log N_2^\mu(\delta, t, \varepsilon, X, f) \leq \lim_{\varepsilon \to 0} \limsup_{t \to \infty} \frac{1}{t} \log N_3^\mu(\delta, t, \varepsilon, X, f).
\]
Similar results for $\liminf$ can be proved similarly. The proof of the lemma is completed.
\end{proof}

By the Lipschitz continuity of a $C^1$ vector field $X\in\mathcal{X}^1(M)$, there is a constant $c>0$ such that for any $x,y\in M\setminus{\rm Sing}(X)$, $d(x,y)<c\|X(x)\|$ implies $\|X(x)\|/2\leq \|X(y)\|\leq 2\|X(x)\|$.

Now we consider the relation between $B^*_1(x, t, \varepsilon, X)$ and $B^*_2(x, t, \varepsilon, X)$. The following lemma can be found in the proof of Lemma 6 in \cite{WANG2024128673} (See also \cite{yang2025rescaledexpansivemeasuresflows}).

\begin{Lemma}\label{lemma3.3}
Fix arbitrary \( a\in(0, c) \)  and \( \tau > \min\{\frac{\log 2}{4L}, r_0\} \). There exists \( \varepsilon > 0 \) such that for any \( t > \tau \) and \( x \in M \setminus \text{Sing}(X) \), there exists  \( F \subseteq B_2^*(x, t, \varepsilon, X) \) with \( \#F \leq 3^{\lfloor\frac{t}{\tau}\rfloor} \) ,\,such that

\[
B_2^*(x, t, \varepsilon, X) \subseteq \bigcup_{y \in F} B_1^*(y, t, 16a, X).
\]
  \end{Lemma}

\noindent
\begin{proof}  We can obtain the conclusion from the proof of Lemma 6 in \cite{WANG2024128673}. For the sake of completeness, we give a sketch of the proof as follows.

Firstly, by the flow box theorem, we can take \( \theta > 0 \) such that for any \( z \in M \setminus\text{Sing}(X) \) and \( s \in [-\theta,\theta] \), we have
\[
d(\phi_s(z), z) < a \|X(z)\|.
\]
Fix a constant \( 0 < b < \min\left\{ \frac{\log 2}{4L}, r_0 \right\} \). Then by Lemma \ref{controloftimereparametrization}, one can choose \( 0 < \varepsilon < a \) such that for any \( y, z \), \( t \in [b, +\infty) \) and \( \alpha \in \text{Rep}[0, t] \), if
\[
d(\phi_{\alpha(s)}(y), \phi_s(z)) < \varepsilon \|X(\phi_{\alpha(s)}(y))\|, \quad \forall \, 0 \leq s \leq t,
\]
then \( |\alpha(s) - s| < \frac{\theta}{4\tau}s \) for any \( s \in [b, t] \).

For any given \(t>\tau \), let \( n = \left\lfloor \frac{t}{\tau} \right\rfloor \) and  \( I_0 \) be the set of all non-negative integer sequences \(\{S(n)\}\) with \( S: \{0, 1, \dots, n\} \to \mathbb{Z}^+ \),\,\( S(0) = 0 \) and \( S(i+1) - S(i)\in\{-1, 0, 1\}\).

 Note that for \(S\in I_0 \), once \( S(0), \ldots, S(i) \) are determined, \( S(i + 1) \) has only three choices. Thus, we have \( \#I_0 \leq 3^n \).

Let \( x \in M \setminus \text{Sing}(X) \) and \( y \in B_2^*(x, t, \varepsilon, X) \). Then there exists \( \alpha \in \text{Rep}[0, t] \) such that
\[
d(\phi_{\alpha(s)}(x), \phi_s(y)) < \varepsilon \|X(\phi_{\alpha(s)}(x))\|, \quad \forall \,0\leq s \leq t.
\]
By Claim 1 in the proof of Lemma 6 in \(\cite{WANG2024128673}\) and the choice of \( \varepsilon \), using Lemma \ref{controloftimereparametrization}, we can get that for any \( 0 \leq s_1 < s_2 \leq t \), if \( b \leq s_2 - s_1 \leq \tau \), then \[
|\alpha(s_2) - s_2 - (\alpha(s_1) - s_1)| \leq \frac{\theta}{4}.
\]
In particular, for \( 0 \leq i \leq n - 1 \), we have
\[
\left| \alpha((i + 1)\tau) - (i + 1)\tau - (\alpha(i\tau) - i\tau) \right| \leq \frac{\theta}{4}. \tag{1}
\]
Let \( S_\alpha(i) = \left\lfloor \frac{\alpha(i\tau) - i\tau}{\frac{\theta}{4}} \right\rfloor \). From (1), we have
\[
S_\alpha(i+1) - S_\alpha(i) \in \{-1, 0, 1\}.
\]
Thus \( S_\alpha \in I_0 \). Let \( I_1 \subseteq I_0 \) be the set of \( S \in I_0 \) satisfying \(\exists\, y \in B_2^*(x, t, \varepsilon, X) \) and \( \alpha \in \text{Rep}[0, t] \) such that
\[
S_{\alpha} = S  \text{   and   }
d(\phi_{\alpha(s)}(x), \phi_s(y)) < \varepsilon \|X(\phi_{\alpha(s)}(x))\|, \quad \forall\,0\leq s \leq t   .
\]

For each \( S \in I_1 \), we can take a corresponding \( y_S \in B_2^*(x, t, \varepsilon, X) \) and \( \alpha_S \in \text{Rep}[0, t] \) respectively. In this way, we obtain a finite set \( F \) of \( B_2^*(x, t, \varepsilon, X) \), and \( \#F \leq \#I_1 \leq \#I_0 \leq 3^n \).

For any \( z \in B_2^*(x, t, \varepsilon, X) \), there exists \( \beta \in \text{Rep}[0, t] \) such that
\[
d(\phi_{\beta(s)}(x), \phi_s(z)) < \varepsilon \|X(\phi_{\beta(s)}(x))\|, \quad \forall\,0\leq s \leq t.
\]
Let \( S_{\beta}(i) = \lfloor\frac{\beta(ir) - ir}{\frac{\theta}{4}}\rfloor \). Then \( S_{\beta} \in I_1 \). Thus, there exists \( y \in F \) and corresponding \( \alpha \in \text{Rep}[0, t] \) such that \( S_{\alpha} = S_{\beta} \).

By Claim 2 in the proof of Lemma 6 in \(\cite{WANG2024128673}\), since \( S_{\alpha} = S_{\beta} \), we have
\[
d(\phi_s(y), \phi_s(z)) \leq 16a \|X(\phi_s(y))\|, \quad \forall \,0\leq s \leq t.
\]
That is, \( z \in B_1^*(y, t, 16a, X) \). Therefore,
\[
B_2^*(x, t, \varepsilon, X) \subseteq \bigcup_{y \in F} B_1^*(y, t, 16a, X).
\]
The proof of Lemma \ref{lemma3.3} is completed.
\end{proof}

\begin{Lemma}\label{lemma3.4}
Let \( \mu \) be a Borel probability measure with \( \mu(\text{Sing}(X)) = 0 \). For any \( f \in C(M, \mathbb{R}) \) and any $\delta \in (0, 1)$, we have
\[
\lim_{\varepsilon \to 0} \limsup_{t \to \infty} \frac{1}{t} \log N_1^\mu(\delta, t, \varepsilon, X, f) = \lim_{\varepsilon \to 0} \limsup_{t \to \infty} \frac{1}{t} \log N_2^\mu(\delta, t, \varepsilon, X, f),
\]
\[
\lim_{\varepsilon \to 0} \liminf_{t \to \infty} \frac{1}{t} \log N_1^\mu(\delta, t, \varepsilon, X, f) = \lim_{\varepsilon \to 0} \liminf_{t \to \infty} \frac{1}{t} \log N_2^\mu(\delta, t, \varepsilon, X, f).
\]
\end{Lemma}

\begin{proof} For any \( x \in M \setminus \text{Sing}(X) \), \( t > 0 \) and \( \varepsilon > 0 \), we have
\[
B_1^*(x, t, \varepsilon, X) \subseteq B_2^*(x, t, \varepsilon, X).
\]

\noindent
Thus, we always have \( N_1^\mu(\delta, t, \varepsilon, X, f) \geq N_2^\mu(\delta, t, \varepsilon, X, f) \), and therefore
\[
\lim_{\varepsilon \to 0} \limsup_{t \to \infty} \frac{1}{t} \log N_1^\mu(\delta, t, \varepsilon, X, f) \geq \lim_{\varepsilon \to 0} \limsup_{t \to \infty} \frac{1}{t} \log N_2^\mu(\delta, t, \varepsilon, X, f),
\]
\[
\lim_{\varepsilon \to 0} \liminf_{t \to \infty} \frac{1}{t} \log N_1^\mu(\delta, t, \varepsilon, X, f) \geq \lim_{\varepsilon \to 0} \liminf_{t \to \infty} \frac{1}{t} \log N_2^\mu(\delta, t, \varepsilon, X, f).
\]

Let \( \sigma > 0 \), \(a\in(0,c)\) and  \( \tau > \min\{\frac{\log 2}{4L}, r_0\} \) be given. By Lemma \ref{lemma3.3}, there exists \( \varepsilon > 0 \), s.t. for any \( t > \tau \) and any \( x \in M \setminus \text{Sing}(X) \), there exists a finite subset \( F_x \subseteq B_2^*(x, t, \varepsilon, X) \) with \( \#F_x \leq 3^{\left\lfloor \frac{t}{\tau} \right\rfloor} \) such that
\[
B_2^*(x, t, \varepsilon, X) \subseteq \bigcup_{y \in F_x} B_1^*(y, t, 16a, X).
\]

By Proposition \ref{boundedvariation}, we can choose \( \varepsilon > 0 \) small enough and \( T>0\) large enough such that for any \( t > T \),  any \( x \in M \setminus \text{Sing}(X) \) and any \( y, z \in B_2^*(x, t, \varepsilon, X) \), we have
\[
\frac{1}{t}\left| \int_0^t f(\phi_s(y)) \, {\rm d}s - \int_0^t f(\phi_s(z)) \, {\rm d}s \right| < \sigma.
\]
Then, for any \( y \in B_2^*(x, t, \varepsilon, X) \), we have \[ \left| \int_0^t f(\phi_s(y)) \, {\rm d}s - \int_0^t f(\phi_s(x)) \, {\rm d}s \right| < \sigma t. \]

Now we always assume \( t>\max\{\tau, T\}\). Let \( F \subseteq M \setminus \text{Sing}(X) \) be a finite set with
\[
\mu\left( \bigcup_{x \in F} B_2^*(x, t, \varepsilon, X) \right) > 1 - \delta.
\]

\noindent
For each \( x \in F \), take \( F_x \subseteq B_2^*(x, t, \varepsilon, X) \) with \(\#F_x \leq 3^{\left\lfloor \frac{t}{\tau} \right\rfloor} \) such that
\[
B_2^*(x, t, \varepsilon, X) \subseteq \bigcup_{y \in F_x} B_1^*(y, t, 16a, X).
\]

\noindent
Then we have
\[
\mu\left( \bigcup_{y \in {\bigcup\limits_{x\in F} F_x}} B_1^*(y, t, 16a, X) \right) \geq \mu\left( \bigcup_{x \in F} B_2^*(x, t, \varepsilon, X) \right) > 1 - \delta.
\]
One can see that
\begin{align*}
    \sum_{y \in {\bigcup\limits_{x\in F} F_x}} e^{\int_0^t f(\phi_s(y)) \, {\rm d}s} &= \sum_{x \in F} \sum_{y \in F_x} e^{\int_0^t f(\phi_s(y)) \, {\rm d}s}\leq \sum_{x \in F} \sum_{y \in F_x} e^{\int_0^t f(\phi_s(x)) \, {\rm d}s + \sigma t}\\
    &\leq \sum_{x \in F} 3^{\left\lfloor \frac{t}{\tau} \right\rfloor} e^{\int_0^t f(\phi_s(x)) \, {\rm d}s + \sigma t}\leq 3^{\frac{t}{\tau}} e^{\sigma t} \sum_{x \in F} e^{\int_0^t f(\phi_s(x)) \, {\rm d}s}.
\end{align*}

Take the infimum over \( F \) , we have
\[
N_1^\mu(\delta, t, 16a, X, f) \leq 3^{\frac{t}{\tau}} e^{\sigma t} N_2^\mu(\delta, t, \varepsilon, X, f). \tag{2}
\]
\noindent
Thus
\[
\limsup_{t \to \infty} \frac{1}{t} \log N_1^\mu(\delta, t, 16a, X, f) \leq \limsup_{t \to \infty} \frac{1}{t} \log N_2^\mu(\delta, t, \varepsilon, X, f) + \frac{\log 3}{\tau} + \sigma.
\]

\noindent
Let \( a \to 0 \) and \( \varepsilon \to 0 \), then
\[
\lim_{\varepsilon \to 0} \limsup_{t \to \infty} \frac{1}{t} \log N_1^\mu(\delta, t, \varepsilon, X, f) \leq \lim_{\varepsilon \to 0} \limsup_{t \to \infty} \frac{1}{t} \log N_2^\mu(\delta, t, \varepsilon, X, f) + \frac{\log 3}{\tau} + \sigma.
\]

\noindent
By the arbitrariness of \( \tau \) and \( \sigma \), let \( \tau \to \infty \) and \( \sigma \to 0 \), we have
\[
\lim_{\varepsilon \to 0} \limsup_{t \to \infty} \frac{1}{t} \log N_1^\mu(\delta, t, \varepsilon, X, f) \leq \lim_{\varepsilon \to 0} \limsup_{t \to \infty} \frac{1}{t} \log N_2^\mu(\delta, t, \varepsilon, X, f).
\]
Similarly, by the inequality (2), we can also get
\[
\lim_{\varepsilon \to 0} \liminf_{t \to \infty} \frac{1}{t} \log N_1^\mu(\delta, t, \varepsilon, X, f) \leq \lim_{\varepsilon \to 0} \liminf_{t \to \infty} \frac{1}{t} \log N_2^\mu(\delta, t, \varepsilon, X, f).
\]
\noindent
This proves that
\[
\lim_{\varepsilon \to 0} \limsup_{t \to \infty} \frac{1}{t} \log N_1^\mu(\delta, t, \varepsilon, X, f) = \lim_{\varepsilon \to 0} \limsup_{t \to \infty} \frac{1}{t} \log N_2^\mu(\delta, t, \varepsilon, X, f),
\]
\[
\lim_{\varepsilon \to 0} \liminf_{t \to \infty} \frac{1}{t} \log N_1^\mu(\delta, t, \varepsilon, X, f) = \lim_{\varepsilon \to 0} \liminf_{t \to \infty} \frac{1}{t} \log N_2^\mu(\delta, t, \varepsilon, X, f).
\]
These finish the proof of Lemma \ref{lemma3.4}.
\end{proof}

Combining Lemma \ref{lemma3.2} and Lemma \ref{lemma3.4} yields the conclusion of Theorem \ref{mainthm1}, which completes the proof of Theorem \ref{mainthm1}.

\section{Katok's Formula for the Rescaled Metric Pressure}\label{sec:Katokformula}
In this section, we will prove Theorem \ref{mainthm2}, which provides Katok's formula for the rescaled metric pressure. Firstly, we prepare several existing results. The following fundamental lemma is firstly given by Ma\~n\'e.

\begin{Lemma}\label{lemma4.1} \cite{Mane_1981} Let \(\mu\) be a probability measure on \(M\), and \(\rho: M \to (0,1)\) be a measurable function. If \(\rho\) is \(\mu\)-integrable, then there exists a countable measurable partition \(\xi\) of \(M\) such that
\begin{enumerate}
\item[1.] \(H_\mu(\xi)=-\sum_{A\in\xi} \mu(A)\log \mu(A) < +\infty\),

\item[2.]  \(\mathrm{diam}\,\xi(x) \leq \rho(x)\) for \(\mu\)-a.e. \(x \in M\),
\end{enumerate}
where \(\xi(x)\) denotes the element of \(\xi\) containing \(x\).
\end{Lemma}

To prove Theorem \ref{mainthm2}, we also need the following famous Shannon-McMillan-Breiman Theorem.

\begin{Lemma}\cite{sun2018}\label{lemma4.2}
Let \(\xi\) be a countable measurable partition with \(H_\mu(\xi)< +\infty\), and let \(T: M \to M\) be an ergodic map. Then for \(\mu\)-a.e. \(x \in M\)
\[
\lim\limits_{n \to \infty} \left( -\frac{1}{n} \log \mu(\xi_n(x)) \right) = h_{\mu}(\phi_\tau, \xi),
\]
where
\[
\xi_n(x) = \xi(x) \cap T^{-1}(\xi(Tx)) \cap \cdots \cap T^{-(n-1)}(\xi(T^{n-1}x)).
\]
\end{Lemma}

The following lemma allows us to discretize time $t$.

\begin{Lemma}\label{lemma4.3}
For any given \(\tau > 0\) and \(\varepsilon > 0\), it holds that
\[
\limsup_{t \to \infty} \frac{1}{t} \log N_1^\mu(\delta, t, \varepsilon, X, f)
\;=\;
\limsup_{n \to \infty} \frac{1}{n\tau} \log N_1^\mu(\delta, n\tau, \varepsilon, X, f),
\]
\[
\liminf_{t \to \infty} \frac{1}{t} \log N_1^\mu(\delta, t, \varepsilon, X, f)
\;=\;
\liminf_{n \to \infty} \frac{1}{n\tau} \log N_1^\mu(\delta, n\tau, \varepsilon, X, f).
\]
\end{Lemma}

\begin{proof} It is obvious that
\[
\limsup_{n \to \infty} \frac{1}{n\tau} \log N_1^\mu(\delta, n\tau, \varepsilon, X, f)
\;\leq\;
\limsup_{t \to \infty} \frac{1}{t} \log N_1^\mu(\delta, t, \varepsilon, X, f),
\]
\[
\liminf_{n \to \infty} \frac{1}{n\tau} \log N_1^\mu(\delta, n\tau, \varepsilon, X, f)
\;\geq\;
\liminf_{t \to \infty} \frac{1}{t} \log N_1^\mu(\delta, t, \varepsilon, X, f).
\]

\noindent
Now, we're going to prove
\[
\limsup_{t \to \infty} \frac{1}{t} \log N_1^\mu(\delta, t, \varepsilon, X, f)
\;\leq\;
\limsup_{n \to \infty} \frac{1}{n\tau} \log N_1^\mu(\delta, n\tau, \varepsilon, X, f).
\]
Recall that
\[
\mathcal{F}_1(t, \varepsilon)
= \left\{
F \subset M \setminus \text{Sing}(X)
:
\mu\left( \bigcup_{x \in F} B_1^*(x, t, \varepsilon, X) \right) > 1 - \delta
\right\}.
\]
Note that \( \mathcal{F}_1(t_1, \varepsilon) \subseteq \mathcal{F}_1(t_2, \varepsilon) \) providing \( t_1 > t_2 \). We have

\begin{align*}
N_1^\mu(\delta, t, \varepsilon, X, f)
&= \inf_{\substack{F \in \mathcal{F}_1(t, \varepsilon) }}
\sum_{x \in F} e^{\int_0^t f(\phi_s(x)) \, {\rm d}s}\leq \inf_{\substack{F \in \mathcal{F}_1\left(\left( \left\lfloor \frac{t}{\tau} \right\rfloor+1\right) \tau, \varepsilon \right) }}
\sum_{x \in F} e^{\int_0^t f(\phi_s(x)) \, {\rm d}s} \\
&\leq \inf_{\substack{F \in \mathcal{F}_1\left(\left( \left\lfloor \frac{t}{\tau} \right\rfloor+1\right) \tau, \varepsilon \right) }}
\sum_{x \in F} e^{\int_0^{( \left\lfloor \frac{t}{\tau} \right\rfloor+1) \tau } f(\phi_s(x)) \, {\rm d}s  }\cdot e^{\tau\|f\|  } \\
&= N_1^\mu\left( \delta, \left(\left\lfloor \frac{t}{\tau} \right\rfloor+1\right) \tau, \varepsilon, X, f \right)
\cdot e^{\tau\|f\|  }.
\end{align*}

\noindent
Thus
\begin{align*}
\limsup_{t \to \infty}  \frac{1}{t}\log N_1^\mu(\delta, t, \varepsilon, X, f)
&\leq \limsup_{t \to \infty} \frac{1}{t} \left[ \log N_1^\mu\left( \delta, \left(\left\lfloor \frac{t}{\tau} \right\rfloor+1\right) \tau, \varepsilon, X, f \right) + \|f\| \cdot \tau \right] \\
&= \limsup_{t \to \infty} \frac{\left(\left\lfloor \frac{t}{\tau} \right\rfloor+1\right) \tau}{t}
\cdot \frac{1}{\left(\left\lfloor \frac{t}{\tau} \right\rfloor+1\right) \tau} \log N_1^\mu\left( \delta, \left(\left\lfloor \frac{t}{\tau} \right\rfloor+1\right) \tau, \varepsilon, X, f \right) \\
&= \limsup_{t \to \infty} \frac{1}{\left(\left\lfloor \frac{t}{\tau} \right\rfloor+1\right) \tau} \log N_1^\mu\left( \delta,\left(\left\lfloor \frac{t}{\tau} \right\rfloor+1\right) \tau, \varepsilon, X, f \right) \\
&= \limsup_{n \to \infty} \frac{1}{n\tau} \log N_1^\mu(\delta, n\tau, \varepsilon, X, f).
\end{align*}
Similarly, we can prove
\[\liminf_{t \to \infty}  \frac{1}{t}\log N_1^\mu(\delta, t, \varepsilon, X, f)\geq\liminf_{n \to \infty} \frac{1}{n\tau} \log N_1^\mu(\delta, n\tau, \varepsilon, X, f), \]
by the fact that

\[N_1^\mu(\delta, t, \varepsilon, X, f)\geq N_1^\mu\left( \delta, \left\lfloor \frac{t}{\tau} \right\rfloor\tau, \varepsilon, X, f\right)
\cdot e^{-\tau\|f\|  }.\]
The lemma is proved.
\end{proof}

\begin{Lemma}\label{lemma4.4}
Let \( X \) be a \( C^1 \) vector field on a compact boundaryless Riemannian manifold \( M \), and let \( \phi_t \) be the flow generated by   X, \( \mu \in \mathcal{E}_X(M) \) with \( \mu(\text{Sing}(X)) = 0 \) and \( \int \log \norm{X(x)} \dmu < \infty \). For any continuous function \( f \in C(X, \mathbb{R}) \), the inequality
\[
\lim_{\varepsilon \to 0} \limsup_{t \to \infty} \frac{1}{t} \log N_i^\mu(\delta, t, \varepsilon, X, f)
\leq h_\mu(\phi_1) + \int f \, {\rm d}\mu,\quad i = 1, 2, 3. \tag{3}
\]
holds.
\end{Lemma}

\begin{proof}
By Theorem \ref{mainthm1}, it suffices to prove that
\[
\lim_{\varepsilon \to 0} \limsup_{t \to \infty} \frac{1}{t} \log N_1^\mu(\delta, t, \varepsilon, X, f)
\leq h_\mu(\phi_1) + \int f \, {\rm d}\mu. \tag{4}
\]

Since \( \mu \in \mathcal{E}_X(M) \), there exists \( \tau \in \mathbb{R} \setminus \{0\} \) such that \( \mu \) is an ergodic measure of \( \phi_\tau \) (the proof is given in \( \cite{pugh1971ergodic}\)). Without loss of generality, we assume \( \tau > 0 \) (otherwise we can consider \( -\tau \)). By Lemma \ref{lemma4.3}, to prove (4), it suffices to prove
\[
\lim_{\varepsilon \to 0} \limsup_{n \to \infty} \frac{1}{n\tau} \log N_1^\mu(\delta, n\tau, \varepsilon, X, f)
\leq \frac{1}{\tau}h_\mu(\phi_\tau) + \int f \, {\rm d}\mu.
\]

\noindent
By Birkhoff's Ergodic Theorem, we have
\[
\lim_{n \to \infty} \frac{1}{n\tau} \int_0^{n\tau} f(\phi_s(x)) \, {\rm d}s
= \int f \, {\rm d}\mu
\]
for \(\mu\text{-a.e. } x \in M\). By Egoroff's Theorem, there exists \( B \subseteq M \setminus \text{Sing}(X) \) such that \( \mu(B) > 1 - {\delta}/{2} \) and
\[
\frac{1}{n\tau} \int_0^{n\tau} f(\phi_s(x)) \, {\rm d}s
\rightrightarrows \int f \, {\rm d}\mu
\quad \text{  on  } B .
\]

\noindent
Thus, for any \(\ \sigma > 0\), there exists \(N_1\), such that
\[
\left| \frac{1}{n\tau} \int_{0}^{n\tau} f(\phi_{s}(x)) \, {\rm d}s - \int f \, {\rm d}\mu \right| < \sigma
\]
for any \(n > N_1\) and $x \in B$.

For any given \( \varepsilon > 0 \), we can take \( \eta > 0 \) such that for any \( x,y \in M \setminus \text{Sing}(X) \), if \( d(x,y) < \eta \|X(x)\| \), then
\(d(\phi_{s}(x), \phi_{s}(y)) < \varepsilon \|X(\phi_{s})\|\) for any \(s\in[0,\tau]\).

Since \( \log \eta \|X(x)\| \) is \( \mu \)-integrable, by Lemma \ref{lemma4.1}, there exists a countable measurable partition \( \xi \) on \( M \) such that \( H_\mu(\xi) < +\infty \) and \( \text{diam}(\xi(x)) < \eta \|X(x)\| \) for \( \mu \)-a.e. \( x \in M \). The set $$\tilde{M}=\{x\in M\setminus{\rm Sing}(X) : diam(\xi(\phi_{n\tau}(x)))< \eta\|X(\phi_{n\tau}(x))\|, \forall n\in\mathbb{Z}\}$$ has full measure.
Then, by the Shannon-McMillan-Breiman Theorem, for \(\mu\text{-a.e. } x \in M\), we have
\[
\lim\limits_{n \to \infty} \left( -\frac{1}{n} \log \mu(\xi_n(x)) \right) = h_{\mu}(\phi_\tau, \xi),
\]
where \(\xi_n(x) = \xi(x) \cap \phi_{-\tau}(\xi(\phi_{\tau}(x))) \cap \cdots \cap \phi_{-(n - 1)\tau}(\xi(\phi_{(n - 1)\tau}(x)))\). By Egoroff's Theorem, there exists \( \tilde{B} \subseteq \tilde{M}\) such that $\mu(\tilde{B})> 1-\delta/2$ and
\[
\lim\limits_{n \to \infty} \left( -\frac{1}{n} \log \mu(\xi_n(x)) \right)
\]
uniformly converges to \( h_{\mu}(\phi_\tau, \xi) \) on \( \tilde{B} \). That is, for any \( \sigma > 0 \), there exists \( N_2 \) such that once \( n > N_2 \), we have
\[
\left| -\frac{1}{n} \log \mu(\xi_n(x)) - h_\mu(\phi_\tau, \xi) \right| < \sigma, \quad \forall x \in \tilde{B}.
\]
Thus, for any \( n > N_2 \), we have
\[
-\frac{1}{n} \log \mu(\xi_n(x)) - h_\mu(\phi_\tau, \xi) < \sigma.
\]

\noindent
\text{Then}
\[
\mu(\xi_n(x)) > e^{-n(h_\mu(\phi_\tau, \xi) + \sigma)}
\]
for any $x\in\tilde{B}$.

For any given \( \sigma > 0 \), take \( N_1, N_2 \) as above, and let \( n > \max\{N_1, N_2\} \). Denote by $\xi_n=\xi\vee\phi_{-\tau}\xi\vee\cdots\vee\phi_{-(n-1)\tau}\xi$. Define
\[
A_{n,\sigma} = \left\{ A \in \xi_n : A\subset \tilde{M}, \mu(A) > e^{-n(h_\mu(\phi_\tau, \xi) + \sigma)} \right\}.
\]
Then \( \tilde{B} \subseteq \bigcup\limits_{A \in A_{n,\sigma}} A \). Note that \( \# A_{n,\sigma} < e^{n(h_\mu(\phi_\tau, \xi) + \sigma)} \).

For each \( A \in A_{n,\sigma} \), once \( A \cap {B} \neq \varnothing \),  we arbitrarily take an element from the intersection, then we get a set \( F\subset B \) with \( \#F < e^{n(h_\mu(\phi_\tau, \xi) + \sigma)} \)
and \[ B \cap \left( \bigcup\limits_{A \in A_{n,\sigma}} A \right) \subset \bigcup\limits_{x \in F} \xi_n(x). \]
Given any \( x \in F \), since
\[
\xi_n(x) = \xi(x) \cap \phi_{-\tau}(\xi(\phi_\tau(x))) \cap \cdots \cap \phi_{-(n-1)\tau}(\xi(\phi_{(n-1)\tau}(x))),
\]
and for each \( 0 \leq k \leq n-1 \), \( \text{diam}(\xi(\phi_{k\tau}(x))) < \eta \| X(\phi_{k\tau}(x)) \| \). And since \( \phi_{k\tau}(\xi_n(x)) \subseteq \xi(\phi_{k\tau}(x)) \), it follows that for any \( y \in \xi_n(x) \),
\[
d(\phi_{k\tau}(y), \phi_{k\tau}(x)) < \eta \| X(\phi_{k\tau}(x)) \|.
\] Furthermore, for all \( 0 \leq s \leq n\tau \), we have
\[
 d(\phi_s(y), \phi_s(x)) < \varepsilon \| X(\phi_{s}(x)) \| .
\]
Thus \( \xi_n(x) \subseteq B_1^*(x, n\tau, \varepsilon, X) \) and then
\[
 \mu\left( \bigcup\limits_{x \in F} B_1^*(x, n\tau, \varepsilon, X) \right) \geq \mu\left( \bigcup\limits_{x \in F} \xi_n(x) \right) \geq \mu\left( B \cap \left( \bigcup\limits_{A \in A_{n,\sigma}} A \right) \right) \geq\mu\left(B\cap\tilde{B}\right)> 1 - \delta.
 \]

Now, we estimate the sum
\[
\sum_{x \in F} e^{\int_{0}^{n \tau} f(\phi_s(x)) \, {\rm d}s}.
\]
Note that for any \( n > N_1 \) and any \(  x \in B \), we have
\( \left| \frac{1}{n\tau} \int_{0}^{n\tau} f(\phi_s(x)) \, {\rm d}s - \int f \, {\rm d}\mu \right| < \sigma \), and  then
\[ \int_{0}^{n\tau} f(\phi_s(x)) \, {\rm d}s < n\tau \left( \int f \, {\rm d}\mu + \sigma \right). \]

\noindent
Since \( \#F < e^{n(h_\mu(\phi_\tau, \xi) + \sigma)} \) and $F\subset B$, we obtain
\[
 \sum\limits_{x \in F} e^{\int_{0}^{n\tau} f(\phi_s(x)) \, {\rm d}s} < e^{n(h_\mu(\phi_\tau, \xi) + \sigma)} \cdot e^{n\tau \left( \int f \, {\rm d}\mu + \sigma \right)}.
\]
This follows that
\begin{align*}
\frac{1}{n\tau} \log N_1^\mu(\delta, n\tau, \varepsilon, X, f)
&< \frac{1}{n\tau} \left[ n(h_\mu(\phi_\tau, \xi) + \sigma) + n\tau \left( \int f \, {\rm d}\mu + \sigma \right) \right] \\
&= \frac{1}{\tau} h_\mu(\phi_\tau, \xi) + \frac{\sigma}{\tau} + \int f \, {\rm d}\mu + \sigma \\
&= \frac{1}{\tau} h_\mu(\phi_\tau, \xi) + \int f \, {\rm d}\mu + \left( \frac{1}{\tau} + 1 \right) \sigma.
\end{align*}
Thus for any \(\sigma > 0 \), there exists \(N =\max\{N_1, N_2\}\) such that if \( n > N \)
\[
\frac{1}{n\tau} \log N_1^\mu(\delta, n\tau, \varepsilon, X, f) < \frac{1}{\tau} h_\mu(\phi_\tau, \xi) + \int f \, {\rm d}\mu + \left( \frac{1}{\tau} + 1 \right) \sigma,
\]
then by the arbitrariness of $\sigma$ and $\tau$, we have
\[
\limsup\limits_{n \to \infty} \frac{1}{n\tau} \log N_1^\mu(\delta, n\tau, \varepsilon, X, f) \leq \frac{1}{\tau} h_\mu(\phi_\tau, \xi) + \int f \, {\rm d}\mu\leq\frac{1}{\tau}h_\mu(\phi_\tau)+\int f{\rm d}\mu=h_\mu(\phi_1)+\int f{\rm d}\mu.
\]
This finishes the proof of Lemma \ref{lemma4.4}.
\end{proof}

Let \(N \geq 3\), and let \(\Omega_{N,n}\) be the set of all mappings from \(\{0, 1, 2, \ldots, n-1\}\) to \(\{1, 2, \ldots, N\}\). Define a distance function \(\rho_{N,n}: \Omega_{N,n} \times \Omega_{N,n} \to [0, 1]\) on \(\Omega_{N,n}\) as follows
\[
\rho_{N,n}(w, \overline{w}) = \frac{1}{n} \sum_{i=0}^{n-1} \bigl( 1 - \delta_{w(i), \overline{w}(i)} \bigr),
\]
where
\[
\delta_{i,j} =
\begin{cases}
0, & i \neq j; \\
1, & i = j.
\end{cases}
\]
The proof of the following lemma can be found in Lemma 4.3.2 in  \(\cite{sun2018}\)   or in Katok's work \(\cite{katok1980lyapunov} .\)

\begin{Lemma}\label{lemma4.5}
 Let \( B(\overline{\omega}, r; N, n) \) denote the ball in \( \Omega_{N,n} \) with center \( \overline{\omega} \) and radius \( r \), that is
\[
B(\overline{\omega}, r; N, n) = \left\{ \omega \in \Omega_{N,n} : \rho_{N,n}(\omega, \overline{\omega}) \leq r \right\}.
\]
Then the following two properties hold:
\begin{enumerate}
\item[1.]
\(
\# B(\overline{\omega}, r; N, n) = \sum_{k=0}^{\lfloor nr \rfloor} C_n^k (N-1)^k,
\)

\item[2.] Let \( P(N, n, r) = \# B(\overline{\omega}, r; N, n) \). When \( 0 \leq r < \dfrac{N-2}{N} \), we have
\[
\lim_{n \to \infty} \frac{\log P(N, n, r)}{n} = r \log(N-1) - r \log r - (1-r) \log(1-r).
\]
\end{enumerate}
\end{Lemma}

\begin{Lemma}\label{lemma4.6}
Suppose that $\mu$ is an ergodic measure of \( \phi_\tau \) for some constant $\tau>0$. Let \( \xi = \{ C_1, C_2, \ldots, C_N \} \) (\( N \geq 3 \)) be a finite measurable partition of \( M \) with \( \mu(\partial \xi) = 0 \). Then for any \( \delta \in (0, 1) \) and \( f \in C(M, \mathbb{R}) \), we have
\[
\lim_{\gamma \to 0} \liminf_{n \to \infty} \frac{1}{n} \log N_1^\mu(\delta, n\tau, \gamma, f, X) \geq \frac{1}{\tau} h_\mu(\phi_\tau, \xi) + \int f \, {\rm d}\mu .
\]
\end{Lemma}
\begin{proof}
Let \( 0 < \varepsilon < \frac{1 - \delta}{2} \) be an arbitrarily given positive constant. Since \( \mu(\partial \xi) = 0 \), there exists a neighborhood \( U \) of \( \partial \xi \) such that \( \mu(U) < \frac{\varepsilon^2}{8} \). Denote by \( \chi_U \) the characteristic function of \( U \). For any \( n \in \mathbb{N} \), define
\[
B_{n, \varepsilon} = \left\{ x \in M : \frac{1}{n} \sum_{i=0}^{n-1} \chi_U(\phi_{i\tau}(x)) < \frac{\varepsilon }{2} \right\}.
\]
\noindent
\textbf{Claim 1. } \(\mu(B_{n, \varepsilon}) > 1 - \frac{\varepsilon}{4} \).

\bigskip

\noindent
\begin{proof} Note that \( \int_M \chi_U(x) \, {\rm d}\mu = \mu(U) < \frac{\varepsilon^2}{8} \). Since \( \mu \) is an invariant measure for \( \phi_\tau \), we have
\[
\int_M \chi_U ( \phi_{i\tau }(x)) \, {\rm d}\mu = \int_M \chi_U(x) \, {\rm d}\mu < \frac{\varepsilon^2}{8}, \quad \forall\, 0 \leq i \leq n-1.
\]
Thus
\[
\sum_{i=0}^{n-1} \int_M \chi_U( \phi_{i\tau }(x)) \, {\rm d}\mu < \frac{n }{8}\varepsilon^2.
\]
On the other hand
\begin{align*}
\sum_{i=0}^{n-1} \int_M \chi_U ( \phi_{i\tau }(x)) \, {\rm d}\mu &= \int_M \left( \sum_{i=0}^{n-1} \chi_U(\phi_{i\tau }(x)) \right) {\rm d}\mu\geq \int_{M \setminus B_{n, \varepsilon}} \left( \sum_{i=0}^{n-1} \chi_U(\phi_{i\tau }(x)) \right) {\rm d}\mu \\&\geq\int_{M \setminus B_{n, \varepsilon}} \frac{n\varepsilon}{2}{\rm d}\mu= \frac{n \varepsilon}{2} \mu(M \setminus B_{n, \varepsilon}).
\end{align*}
\noindent
Thus \( \frac{n\varepsilon}{2} \mu(M \setminus B_{n, \varepsilon}) < \frac{n}{8} \varepsilon^2 \), then \( \mu(M \setminus B_{n, \varepsilon}) < \frac{\varepsilon}{4} \).
\end{proof}

Also, for \( \mu \)-a.e. \( x \in M \)
\[
\lim_{n \to \infty} \frac{1}{n\tau} \int_0^{n\tau} f(\phi_{ t}(x)) \, {\rm d}t= \int f \, {\rm d}\mu.
\]

\noindent
By Egoroff's Theorem, there exists \( B \subseteq M \) such that \( \mu(B) > 1 - \frac{\varepsilon}{4} \) and \( \frac{1}{n\tau} \int_0^{n\tau} f(\phi_{ t}(x)) \, {\rm d}t \) converges uniformly to \( \int f \, {\rm d}\mu \) as \( n \to \infty \) on \( B \). That is, there exists \( N_1 \) such that for any \( n > N_1 \), we have
\[
\left| \frac{1}{n\tau} \int_0^{n\tau} f(\phi_{ t}(x)) \,{\rm d} t - \int f \, {\rm d}\mu \right| < \varepsilon
\]
for any \( x \in B \).

Let \( A_{n, r}^- = \left\{ A \in  \xi_n : \mu(A) < e^{-n(h_\mu(\phi_\tau, \xi) - \varepsilon)} \right\} \), where \( \xi_n = \xi \vee \phi_{\tau}^{-1} \xi \vee \cdots \vee \phi_{\tau}^{-(n-1)} \xi \). Similarly, by the Shannon-McMillan-Breiman Theorem, there exists \( N_2 \) such that for any \( n > N_2 \), we have
\[
\mu\left( \bigcup_{A \in A^-_{n, r}} A \right) > 1 - \frac{\varepsilon}{4}.
\]

Fix arbitrary $f\in C(M, \mathbb{R})$. We can take \( \gamma_0 > 0 \) such that if \( d(x, y) < \gamma_0 \|X(x)\| \), then the followings hold:
\begin{enumerate}
\item[1.]
    \(
    |f(x) - f(y)| < \varepsilon,
    \)
\item[2.]
    Either  \( \{x, y\} \subseteq U \) or  \( \{x, y\} \subseteq C_i \) for some \( i = 1, 2, \ldots, N \).
\end{enumerate}

The first item in above follows from the uniform continuity of \( f \). If there were no such \( \gamma_0 \) such that the second item holds, then we could take \( x_k, y_k \) containing in different \( C_i \) respectively, and \( x_k, y_k \) not both containing in \( U \), with \( d(x_k, y_k) \to 0 \). By the compactness of $M$, without loss of generality, assume that \( x_k, y_k \) converge (to the same point). Then their limit point is an element in \( \partial{\xi} \). When \( k \) is sufficiently large, \( x_k, y_k \) would both belong to \( U \), which is a contradiction.

Fix \( n > \max\{N_1, N_2\} \) , \( 0 < \gamma < \gamma_0 \). Assume that there exists a finite set \( F \subseteq M \setminus \text{Sing}(X) \) satisfying
\[
\mu\left( \bigcup_{x \in F} B_1^*(x, n\tau, \gamma, X) \right) > 1 - \delta.
\]
Then
\[
\mu\left( \left( \bigcup_{x \in F} B_1^*(x, n\tau, \gamma, X) \right) \cap \left( \bigcup_{A \in \mathcal{A}^-_{n, r}} A \right) \cap B_{n, \varepsilon} \cap B \right) > 1 - \delta - \frac{\varepsilon}{4} - \frac{\varepsilon}{4} - \frac{\varepsilon}{4}> \frac{1 - \delta}{4}.
\]

Let
\(
Q = \left( \bigcup\limits_{x \in F} B_1^*(x, n\tau, \gamma, X) \right) \cap \left( \bigcup\limits_{A \in A^-_{n, r}} A \right) \cap B_{n, \varepsilon} \cap B \), and set
\[
\mathcal{E} = \left\{ A \in \xi_{n} : A \cap Q \neq \emptyset \right\}.
\]
Then \( \mathcal{E} = \left\{ A \in A^-_{n,r} : A \cap \left( \bigcup\limits_{x \in F} B_1^*(x, n\tau, \gamma, X) \right) \cap B_{n,\varepsilon} \cap B \neq \emptyset \right\} \). For \( A \in \mathcal{E} \), since \( A \in {A}_{n,r}^- \), we have \( \mu(A) < e^{-n(h_\mu(\phi_\tau, \xi) - \varepsilon)} \), while \( \mu\left( \bigcup\limits_{A \in \mathcal{E}} A \right) > \mu(Q) > \frac{1 - \delta}{4} \). Thus we have
\[
\#\mathcal{E} > \frac{1 - \delta}{4} e^{n(h_\mu(\phi_\tau, \xi) - \varepsilon)}.
\]

For any \( x \in M \), there exists a unique map \( \omega_x: \{0, 1, \ldots, n - 1\} \to \{1, 2, \ldots, N\} \) such that \( \phi_{i\tau}(x)\in C_{\omega_x(i)} \). It is easy to see that if
\[
y \in \xi_n(x) = \xi(x) \cap \phi_{-\tau}(\xi(\phi_\tau(x))) \cap \cdots \cap \phi_{-(n-1)\tau}(\xi(\phi_{(n-1)\tau}(x))),
\]
then we have \( \omega_x = \omega_y \), where \( \xi(x) \) is the element in \( \xi \) containing \( x \).

Let \( F' = \left\{ x \in F : B_1^*(x, n\tau, \gamma, X) \cap B_{n,\varepsilon} \cap B \neq \emptyset \right\} \). For any \( x \in F' \), let
\[
\mathcal{E}_x = \left\{ A \in {A}_{n,r}^- : A \cap B_1^*(x, n\tau, \gamma, X) \cap B_{n,\varepsilon} \cap B \neq \emptyset \right\}. \tag{4}
\]
Then \( \mathcal{E} \subseteq \bigcup\limits_{x \in F'} \mathcal{E}_x \).

We have the following estimation for the number of elements in \( \mathcal{E}_x \).

\bigskip
\noindent
\(\textbf{Claim 2. }\) {\it For any \( x \in F' \), we have \( \#\mathcal{E}_x \leq P(N, n, \frac{\varepsilon}{2}) \).}

\begin{proof}
Let \( A \in \mathcal{E}_x \). Note that for any \( y, z \in A \), we  have \( \omega_y = \omega_z \), thus we can denote this common sequence as \( \omega_A \). Since \( A \in \mathcal{E}_x \), there exists \( y \in A \cap B_1^*(x, n\tau, \gamma, X) \cap B_{n, \varepsilon} \cap B \), and \( \omega_y = \omega_A \).

Since \( y \in B_1^*(x, n\tau, \gamma, X) \), then for all \( 0 \leq i \leq n - 1 \), we have
\[
d(\phi_{i\tau}(x), \phi_{i\tau}(y)) < \gamma \| X(\phi_{i\tau}(x)) \| < \gamma_0 \| X(\phi_{i\tau}(x)) \|.
\]
Thus, \( \phi_{i\tau}(x) \) and \( \phi_{i\tau}(y)\, (0\leq i\leq n-1)\) either both belong to the same element of \(\xi\) or both belong to \( U \). That is, if \( \omega_x(i) \neq \omega_y(i) \), then \( \chi_U(\phi_{i\tau}(x)) = \chi_U(\phi_{i\tau}(y)) = 1 \). Thus we have
\[
1 - \delta_{\omega_x(i), \omega_y(i)} \leq \chi_U(\phi_{i\tau}(x))
\]
for any $0\leq i\leq n-1$. Then
\[
\rho_{N,n}(\omega_x, \omega_y) = \frac{1}{n} \sum_{i = 0}^{n - 1} (1 - \delta_{\omega_x(i), \omega_y(i)}) \leq \frac{1}{n} \sum_{i = 0}^{n - 1} \chi_U(\phi_{i\tau}(x)) < \frac{\varepsilon}{2}.
\]
This proves that \( \rho_{N,n}(\omega_x, \omega_A) < \frac{\varepsilon}{2} \), and then \( \omega_A \in B\left( \omega_x, \frac{\varepsilon}{2}; N, n \right) \).

Therefore, when \(A\in \mathcal{E}_x \), we always have \( \omega_A \in B\left( \omega_x, \frac{\varepsilon}{2}; N, n \right) \). Note that for any $A, B\in \xi_n$, \( A=B \) if and only if \( \omega_A = \omega_B \) . We can see that \( \#\mathcal{E}_x \leq \# B\left( \omega_x, \frac{\varepsilon}{2}; N, n \right) = P\left( N, n, \frac{\varepsilon}{2} \right) \). This finishes the proof of the claim.
\end{proof}

By the claim, we have
\[
\frac{1 - \delta}{4} e^{n(h_\mu(\phi_\tau, \xi) - \varepsilon)} < \#\mathcal{E} = \#\left( \bigcup_{x \in F'} \mathcal{E}_x \right) \leq \# F' \cdot P\left( N, n, \frac{\varepsilon}{2} \right).
\]
Hence we have
\[
\# F' \geq \frac{1 - \delta}{4} e^{n(h_\mu(\phi_\tau, \xi) - \varepsilon)} \cdot P\left( N, n, \frac{\varepsilon}{2} \right)^{-1}.
\]

For any \( x \in F' \), since \( B_1^*(x, n\tau, \gamma, X) \cap B_{n, \varepsilon} \cap B \neq \emptyset \), there exists
\(
y \in B_1^*(x, n\tau, \gamma, X) \cap B_{n, \varepsilon} \cap B.
\)
For any $t\in[0, n\tau]$, because
\[
d(\phi_t(x), \phi_t(y)) < \gamma\| X(\phi_t(x)) \| < \gamma_0 \| X(\phi_t(x)) \|,
\]
we have \(
|f(\phi_t(x)) - f(\phi_t(y))| < \varepsilon.
\)

Note that for any \( y \in B \) and any \( n > N_1 \), we have
\[
\left| \frac{1}{n\tau} \int_0^{n\tau} f(\phi_t(y)) {\rm d}t - \int f {\rm d}\mu \right| < \varepsilon.
\]
Thus we have
\[
\begin{aligned}
&\left| \frac{1}{n\tau} \int_0^{n\tau} f(\phi_t(x)) {\rm d}t - \int f {\rm d}\mu \right| \\
\leq& \frac{1}{n\tau} \left( \left| \int_0^{n\tau} f(\phi_t(x)) {\rm d}t - \int_0^{n\tau} f(\phi_t(y)) {\rm d}t \right| + \left| \int_0^{n\tau} f(\phi_t(y)) {\rm d}t-  n\tau\int f {\rm d}\mu \right| \right)\\
<& 2\varepsilon.
\end{aligned}
\]
And then
\[
\int_0^{n\tau} f(\phi_t(x)) {\rm d}t> n\tau \left( \int f {\rm d}\mu - 2\varepsilon \right).
\]
This follows that
\[
\begin{aligned}
\sum_{x \in F} e^{\int_0^{n\tau} f(\phi_t(x)) {\rm d}t} &\geq \sum_{x \in F'} e^{\int_0^{n\tau} f(\phi_t(x)) {\rm d}t} \geq \sum_{x \in F'} e^{n\tau \left( \int f {\rm d}\mu - 2\varepsilon \right)} \geq \# F' \cdot e^{n\tau \left( \int f {\rm d}\mu - 2\varepsilon \right)} \\
&\geq \frac{1 - \delta}{4} e^{n(h_\mu(\phi_\tau, \xi) - \varepsilon)} \cdot P\left( N, n, \frac{\varepsilon}{2} \right)^{-1} \cdot e^{n\tau \left( \int f {\rm d}\mu - 2\varepsilon \right)}.
\end{aligned}
\]

This proves that if \( n > \max\{N_1, N_2\}, \,0<\gamma<\gamma_0 \), then we have
\[
N_1^\mu(\delta, n\tau, \gamma, X, f) \geq \frac{1 - \delta}{4} e^{n(h_\mu(\phi_\tau, \xi) - \varepsilon)} P\left( N, n, \frac{\varepsilon}{2} \right)^{-1} e^{n\tau \left( \int f {\rm d}\mu - 2\varepsilon \right)}.
\]

\noindent
It follows that
\[
\begin{aligned}
&\liminf_{n \to \infty} \frac{1}{n\tau} \log N_1^\mu(\delta, n\tau, \gamma, X, f) \\
\geq& \liminf_{n \to \infty} \frac{1}{n\tau} \left( n(h_\mu(\phi_\tau, \xi) - \varepsilon) + n\tau \left( \int f {\rm d}\mu -2\varepsilon \right)+\log(1 - \delta) - \log 4 - \log P\left( N, n, \frac{\varepsilon}{2} \right) \right) \\
=& \frac{1}{\tau} h_\mu(\phi_\tau, \xi) + \int f {\rm d}\mu - \frac{\varepsilon}{\tau} - 2\varepsilon -\frac{1}{\tau} \left(\frac{\varepsilon}{2} \log(N - 1) - \frac{\varepsilon}{2} \log \frac{\varepsilon}{2} - \left( 1 - \frac{\varepsilon}{2} \right) \log \left( 1 - \frac{\varepsilon}{2} \right)\right).
\end{aligned}
\]

\noindent
Thus
\[
\begin{aligned}
\lim_{\gamma\to 0} \liminf_{n \to \infty} \frac{1}{n\tau} \log N_1^\mu(\delta, n\tau, \gamma, X, f) &\geq \frac{1}{\tau} h_\mu(\phi_\tau, \xi) + \int f {\rm d}\mu - \frac{\varepsilon}{\tau} - 2\varepsilon -\frac{\varepsilon}{2\tau} \log(N - 1) \\
&+ \frac{\varepsilon}{2\tau} \log \frac{\varepsilon}{2} +\frac{1}{\tau} \left( 1 - \frac{\varepsilon}{2} \right) \log \left( 1 - \frac{\varepsilon}{2} \right).
\end{aligned}
\]

\noindent
By letting \( \varepsilon \to 0 \), we can get
\[
\lim_{\gamma\to 0} \liminf_{n \to \infty} \frac{1}{n\tau} \log N_1^\mu(\delta, n\tau, \gamma, X, f) \geq \frac{1}{\tau} h_\mu(\phi_\tau, \xi) + \int f {\rm d}\mu.
\]
This finishes the proof of the lemma.
\end{proof}

Note that \( h_\mu(\phi_\tau) = \sup\limits_\xi h_\mu(\phi_\tau, \xi) \), where \( \xi \) is a finite measurable partition with \( \mu(\partial \xi) = 0 \). Without loss of generality, assume \( \#\xi \geq 3 \). Applying Lemma \ref{lemma4.3} and taking the supremum over all \( \xi \), we have the following: for any \( 0 < \delta < 1 \), one has
\[
\lim_{\varepsilon \to 0} \liminf_{t \to \infty} \frac{1}{t} \log N_1^\mu(\delta, t, \varepsilon, X, f) \geq h_\mu(\phi_1) + \int f {\rm d}\mu.
\]
Assuming that $\mu$ is an ergodic invariant measure of $\phi$ with $\mu({\rm Sing}(X)=0$ and $\log\|X(x)\|$ is integrable, then by Lemma \ref{lemma4.4}, we can see that for any $\delta\in(0,1)$,
\[
\lim_{\varepsilon \to 0} \liminf_{t \to \infty} \frac{1}{t} \log N_1^\mu(\delta, t, \varepsilon, X, f) = h_\mu(\phi_1) + \int f {\rm d}\mu=\lim_{\varepsilon \to 0} \limsup_{t \to \infty} \frac{1}{t} \log N_1^\mu(\delta, t, \varepsilon, X, f).
\]
Combining with Theorem \ref{mainthm1}, we get the conclusions of Theorem \ref{mainthm2} and complete the proof of Theorem \ref{mainthm2}.

\section{Rescaled Topological Pressure }\label{sec:tp}
In this section, we prove Theorem \ref{theorem 1.3} and Theorem \ref{theorem1.4:vp}. Similar to Lemma \ref{lemma3.2}, we can prove the following lemma by using Lemma \ref{lemma3.1}.

\begin{Lemma}\label{lemma5.1}
Let $f$ be a continuous function. We have $P_2^*(X, f)=P_3^*(X, f)$ and $Q_2^*(X, f)=Q_3^*(X ,f)$.
\end{Lemma}

\begin{proof}
Let $K\subset M\setminus{\rm Sing}(X)$ be a compact set. Denote by $\mathcal{G}_i{(t, K)} $ the set $  \{ F\subset K : F \text{ is a finite rescaled i-}\\(t, \varepsilon, K)  \text{-} \text{spanning} \text{ set} \}$ and $\mathcal{H}_i{(t, K)} = \{ E\subset K : E \text{ is a finite rescaled i-} (t, \varepsilon, K)\text{-separating set} \}$. Given any $\lambda > 0$, if we choose $\varepsilon$ sufficiently small and $t$ sufficiently large, then we can check that
\(
\mathcal{G}_3((1-\lambda)t, K) \supset \mathcal{G}_2(t, K)
\)  and \(
\mathcal{H}_3((1-\lambda)t, K) \subset \mathcal{H}_2(t, K)
\) by the fact that $B_2^*(x, t, \varepsilon, X) \subseteq B_3^*(x, (1 - \lambda)t, \varepsilon, X) (\forall x\in M\setminus{\rm Sing}(X))$.
Then we can get
\[
N_{2,t}^*(X, f, \varepsilon, K) \geq e^{-\lambda t \|f\|} N_{3,(1-\lambda)t}^*(X, f, \varepsilon, K),
\]
\[
Z_{2,t}^*(X, f, \varepsilon, K) \geq e^{-\lambda t \|f\|} Z_{3,(1-\lambda)t}^*(X, f, \varepsilon, K),
\]
and
\[
\frac{1}{t} \log N_{2,t}^*(X, f, \varepsilon, K) \geq -\lambda \|f\| + \frac{1}{t} \log N_{3,(1-\lambda)t}^*(X, f, \varepsilon, K),
\]
\[
\frac{1}{t} \log Z_{2,t}^*(X, f, \varepsilon, K) \geq -\lambda \|f\| + \frac{1}{t} \log Z_{3,(1-\lambda)t}^*(X, f, \varepsilon, K).
\]
Taking limits $t\to+\infty$ and $\varepsilon\to 0^+$, and letting $\lambda \to 0$, we have
\[
\lim_{\varepsilon \to 0} \limsup_{t \to \infty} \frac{1}{t} \log N_{2,t}^*(X, f, \varepsilon, K) \geq \lim_{\varepsilon \to 0} \limsup_{t \to \infty} \frac{1}{t} \log N_{3,t}^*(X, f, \varepsilon, K),
\]
\[
\lim_{\varepsilon \to 0} \liminf_{t \to \infty} \frac{1}{t} \log Z_{2,t}^*(X, f, \varepsilon, K) \geq \lim_{\varepsilon \to 0} \liminf_{t \to \infty} \frac{1}{t} \log Z_{3,t}^*(X, f, \varepsilon, K).
\]
We can proceed from the fact that $B_3^*(x, t, \varepsilon, X) \subseteq B_2^*(x, (1 - \lambda)t, \varepsilon, X) (\forall x\in M\setminus{\rm Sing}(X))$ to prove the reverse inequality and get
\[
\lim_{\varepsilon \to 0} \limsup_{t \to \infty} \frac{1}{t} \log N_{2,t}^*(X, f, \varepsilon, K) =\lim_{\varepsilon \to 0} \limsup_{t \to \infty} \frac{1}{t} \log N_{3,t}^*(X, f, \varepsilon, K),
\]
\[
\lim_{\varepsilon \to 0} \liminf_{t \to \infty} \frac{1}{t} \log Z_{2,t}^*(X, f, \varepsilon, K) = \lim_{\varepsilon \to 0} \liminf_{t \to \infty} \frac{1}{t} \log Z_{3,t}^*(X, f, \varepsilon, K).
\]
By taking the supremum over all compact subsets $K\subset M\setminus{\rm Sing}(x)$ in the above equation, we obtain the equalities $P_2^*(X, f) = P_3^*(X, f)$ and $Q_2^*(X, f) = Q_3^*(X, f)$. This finishes the proof of the lemma.
\end{proof}

We say that a rescaled i-$(t, \varepsilon, K)$-separating set $E$ is {\it maximal} if $E\cup\{x\}$ is not a rescaled i-$(t, \varepsilon, K)$-separating set for any $x\in K\setminus E$. By the compactness of $K$, it is easy to see that every i-$(t, \varepsilon, K)$-separating set can be extended to be a finite maximal i-$(t, \varepsilon, K)$-separating set. Thus, we have
$$Z_{i,t}^*(X, f, \varepsilon,K) = \sup\left\{ \sum_{x \in E} e^{\int_0^t f(\phi_s (x)) {\rm d}s} : E \text{ is a maximal rescaled i-}(t,\varepsilon,K)\text{-separating set} \right\}$$
for any $i=1,2,3$.

\begin{Lemma}\label{lemma5.2}
Let $\varepsilon\in(0, c)$. If $E$ is a maximal rescaled $1$-$(t, \varepsilon/2, K)$-separating set, then $E$ is a $1$-$(t, \varepsilon, K)$-spanning set.
\end{Lemma}

\begin{proof}
Let $E$ be a maximal rescaled 1-$(t, \varepsilon/2, K)$-separating set and $y\in K\setminus E$. We will show that $y\in\bigcup\limits_{x\in E}B_1^*(x, t, \varepsilon, X)$. Since $E\cup\{y\}$ is not a 1-$(t, \varepsilon/2, K)$-separating set, we can find $x\in E$ such that either $y\in B_1^*(x, t, \varepsilon/2, X)$ or $x\in B_1^*(y, t, \varepsilon/2, X)$. If $y\in B_1^*(x, t, \varepsilon/2, X)$, then $y\in\bigcup\limits_{x\in E}B_1^*(x, t, \varepsilon, X)$ obviously. If $x\in B_1^*(y, t, \varepsilon/2, X)$, we have
$$d(\phi_s(x),\phi_s(y))<\frac{\varepsilon}{2}\|X(\phi_s(y))\|\leq\varepsilon\|X(\phi_s(x))\|$$
for any $s\in[0,t]$. Hence $y\in B_1^*(x, t, \varepsilon, X)$ and then $y\in\bigcup\limits_{x\in E}B_1^*(x, t, \varepsilon, X)$. This proves that $y\in\bigcup\limits_{x\in E}B_1^*(x, t, \varepsilon, X)$ when $y\in K\setminus E$, and then $E$ is a 1-$(t, \varepsilon, K)$-spanning set.
\end{proof}

\begin{Lemma}\label{lemma5.3}
For any $\lambda\in(0,1)$ and $b\in(0, r_0)$, there is $\varepsilon_0>0$ such that for any $\varepsilon\in(0, \varepsilon_0)$ and any $t>b$ and compact subset $K\subset M\setminus{\rm Sing}(X)$, if $E$ is a maximal rescaled $2$-$(t, \varepsilon/2, K)$-separating set, then $E$ is a $2$-$((1-\lambda)t, \varepsilon, K)$-spanning set.
\end{Lemma}
\begin{proof}
Let $\lambda\in(0,1)$ and $b\in(0, r_0)$ be given. By Lemma \ref{controloftimereparametrization}, we can find $\varepsilon_0>0$ such that for any $x\in M\setminus{\rm Sing}(X)$, $t>b$ and any $\alpha\in {\rm Rep}[0, t]$, if
$$d(\phi_{\alpha(s)}(x), \phi_s(y))<\varepsilon_0\|X(\phi_{\alpha(s)}(x))\|$$
holds for every $s\in[0, t]$, then $|\alpha(s)-s| \leq \lambda s$ for any $s\in[b, t]$. Without loss of generality, we assume that $0<\varepsilon_0<c$.

Now let $\varepsilon\in(0, \varepsilon_0)$ and $K\subset M\setminus{\rm Sing}(X)$ be a compact set. Suppose that $E$ is a maximal rescaled 2-$(t, \varepsilon/2, K)$-separating set. Then for any $y\in K\setminus E$, there is $x\in E$ such that either $y\in B_2^*(x, t, \varepsilon/2, X)$ or $x\in B_2^*(y, t, \varepsilon/2, X)$. If $y\in B_2^*(x, t, \varepsilon/2, X)$ then it is easy to see that
$$y\in B_2^*(x, t, \varepsilon, X)\subset B_2^*(x, (1-\lambda)t, \varepsilon, X)\subset\bigcup_{x\in E} B_2^*(x, (1-\lambda)t, \varepsilon, X).$$
If $x\in B_2^*(y, t, \varepsilon/2, X)$, then there exists a time reparametrization $\alpha\in {\rm Rep}[0, t]$ such that
$$d(\phi_{\alpha(s)}(y), \phi_s(x))<\frac{\varepsilon}{2}\|X(\phi_{\alpha(s)}(y))\|$$
for any $s\in[0,t]$. By the choice of $\varepsilon_0$, we have $\alpha(t)\geq (1-\lambda)t$. Denote by $\beta=\alpha^{-1}$, then we have $\beta\in {\rm Rep}[0, (1-\lambda)t]$ and
$$d(\phi_s(y), \phi_{\beta(s)}(x))<\frac{\varepsilon}{2}\|X(\phi_s(y))\|\leq \varepsilon\|X(\phi_{\beta(s)}(x))\|$$
for any $s\in[0, (1-\lambda)t]$. This proves that
$$y\in B_2^*(x, (1-\lambda)t, \varepsilon, X)\subset \bigcup_{x\in E} B_2^*(x, (1-\lambda)t, \varepsilon, X).$$ This proves that in both cases we have
$$y\in\bigcup_{x\in E} B_2^*(x, (1-\lambda)t, \varepsilon, X),$$
and finishes the proof of the lemma.
\end{proof}

\begin{Lemma}\label{lemma5.4}
Let $f\in C(M, \mathbb{R})$ be given. We have $P_2^*(X,f)\leq P_1^*(X, f)\leq Q_1^*(X, f)$ and $P_2^*(X, f)\leq Q_2^*(X, f)\leq Q_1^*(X,f)$.
\end{Lemma}
\begin{proof}
Fix a compact subset $K\subset M\setminus{\rm Sing}(X)$. Since
\[
B_1^*(x, t, \varepsilon, X) \subseteq B_2^*(x, t, \varepsilon, X)
\]
for any \( x \in M \setminus \text{Sing}(X) \), \( t > 0 \) and \( \varepsilon > 0 \), we know that every 1-$(t, \varepsilon, K)$-spanning set is 2-$(t, \varepsilon, K)$-spanning and every 2-$(t, \varepsilon, K)$-separating set is 1-$(t, \varepsilon, K)$-separating. Hence we have $N_{1, t}^*(X, f, \varepsilon, K)\geq N_{2, t}^*(X, f, \varepsilon, K)$ and $Z
_{1, t}^*(X, f, \varepsilon, K)\geq Z_{2, t}^*(X, f, \varepsilon, K)$. Then by the definitions we can get $P_1^*(X, f)\geq P_2^*(X, f)$  and $Q_1^*(X, f)\geq Q_2^*(X, f)$ directly.

For every compact subset $K\subset M\setminus{\rm Sing}(X)$ and $\varepsilon>0$ small enough, by Lemma \ref{lemma5.2}, we know that every maximal rescaled 1-$(t, \varepsilon/2, K)$-separating set is rescaled 1-$(t, \varepsilon, K)$-spanning. Thus we have $N_{1, t}^*(X, f, \varepsilon, K)\leq Z_{1, t}^*(X, f, \varepsilon, K)$, and then by definitions of $P_1^*(X, f)$ and $Q_1^*(X, f)$, we have $P_1^*(X, f)\leq Q_1^*(X, f)$. This proves that $P_2^*(X, f)\leq P_1^*(X, f)\leq Q_1^*(X,f)$.

For every compact subset $K\subset M\setminus{\rm Sing}(X)$, $\lambda\in(0,1)$, by Lemma \ref{lemma5.3} we know that if we choose $\varepsilon>0$  small enough and $t>0$ large enough, then every maximal rescaled 2-$(t,\varepsilon/2, K)$-separating set is rescaled 2-$((1-\lambda)t, \varepsilon, K)$-spanning. Thus, for $\varepsilon>0$ small enough and $t$ large enough, we have
\[
\begin{aligned}
N_{2, (1-\lambda)t}^*(X, f, \varepsilon, K) & =\inf\left\{\sum_{x \in E} e^{\int_0^{(1-\lambda)t} f(\phi_s (x)) {\rm d}s} : E \text{ is a rescaled 2-} (t,\varepsilon,K)\text{-spanning set} \right\}\\
& \leq\inf\left\{\sum_{x \in E} e^{\int_0^{(1-\lambda)t} f(\phi_s (x)) {\rm d}s} : E \text{ is a maximal rescaled  2-} (t,\varepsilon,K)\text{-separating set} \right\}\\
&\leq e ^{\lambda \|f\| t}\inf\left\{\sum_{x \in E} e^{\int_0^{t} f(\phi_s (x)) {\rm d}s} : E \text{ is a maximal rescaled  2-} (t,\varepsilon,K)\text{-separating set} \right\}\\
&\leq e^{\lambda \|f\|t}Z_{2, t}^*(X, f, \varepsilon, K).
\end{aligned}
\]
Then we have
\[
\frac{1}{t} \log Z_{2,t}^*(X, f, \varepsilon, K) \geq -\lambda \|f\| + \frac{1}{t} \log N_{2,(1-\lambda)t}^*(X, f, \varepsilon, K).
\]
By taking limits $t\to+\infty$ and $\varepsilon\to 0^+$, and letting $\lambda \to 0$, we have
\[
\lim_{\varepsilon \to 0} \limsup_{t \to \infty} \frac{1}{t} \log Z_{2,t}^*(X, f, \varepsilon, K) \geq \lim_{\varepsilon \to 0} \limsup_{t \to \infty} \frac{1}{t} \log N_{2,t}^*(X, f, \varepsilon, K).
\]
By taking the supremum over all compact subsets $K\subset M\setminus{\rm Sing}(X)$, we have $Q_2^*(X, f)\geq P_2^*(X, f)$. This proves that $P_2^*(X, f)\leq Q_2^*(X, f)\leq Q_1^*(X, f)$.
\end{proof}

\begin{Lemma}\label{lemma5.5}
Let $f\in C(M, \mathbb{R})$ be given. We have $Q_1^*(X,f)\leq P_2^*(X, f)$.
\end{Lemma}

\begin{proof}
Let \( \sigma > 0 \), \(a\in(0,c/16)\) and  \( \tau > \min\{\frac{\log 2}{4L}, r_0\} \) be some fixed constants. By Lemma \ref{lemma3.3}, there exists \( \varepsilon_1> 0 \), s.t. for any \( t > \tau \) and any \( x \in M \setminus \text{Sing}(X) \), there exists a finite subset \( F_x \subseteq B_2^*(x, t, \varepsilon_1, X) \) with \( \#F_x \leq 3^{\left\lfloor \frac{t}{\tau} \right\rfloor} \) such that
\[
B_2^*(x, t, \varepsilon_1, X) \subseteq \bigcup_{y \in F_x} B_1^*(y, t, 16a, X).
\]
By Proposition \ref{boundedvariation}, we can choose \( \varepsilon_2> 0 \) sufficiently small and \( T>0\) large enough such that for any \( t > T \),  any \( x \in M \setminus \text{Sing}(X) \) and any \( y, z \in B_2^*(x, t, \varepsilon_2, X) \), we have
\[
\frac{1}{t}\left| \int_0^t f(\phi_s(y)) \, {\rm d}s - \int_0^t f(\phi_s(z)) \, {\rm d}s \right| < \sigma.
\]
Then, for any \( y \in B_2^*(x, t, \varepsilon, X) \), we have \[ \left| \int_0^t f(\phi_s(y)) \, {\rm d}s - \int_0^t f(\phi_s(x)) \, {\rm d}s \right| < \sigma t. \]
In the following, we always let $0<\varepsilon<\min \{\varepsilon_1, \varepsilon_2\}$ and $t>\max\{\tau, T\}$.

Given a compact \( K \subset M\setminus{\rm Sing}(X) \), let \( E \) be a rescaled \( 1\text{-}(t, 64a, K) \)-separating set and $F$ be a rescaled 2-$(t, \varepsilon, K)$-spanning set. By applying Lemma \ref{lemma3.3}, we can choose $F_x\subset B_2^*(x, t, \varepsilon, X)$ with $\#F_x \leq 3^{\left\lfloor \frac{t}{\tau} \right\rfloor}$ such that
$$B_2^*(x, t, \varepsilon, X)\subset \bigcup_{y\in F_x}B_1^*(x, t, 16a, X)$$
for any $x\in F$. Denote by $G=\bigcup\limits_{x\in F} F_x$. It is easy to see that
\[K \subset \bigcup_{x \in F} B_2^*(x, t, \varepsilon, X) \subset \bigcup_{x \in F} \bigcup_{y \in F_x} B_1^*(y, t, 16a, X) = \bigcup_{y \in G} B_1^*(y, t, 16a, X).
\]

\bigskip
\noindent\textbf{Claim. } {\it For any $y\in G$, there exists at most one point $q\in E$ such that $q\in B_1^*(y, t, 16a, X)$.}

\begin{proof} Suppose that there exist \( p \neq q \in E \) such that \( p, q \in B_1^*(y, t, 16a, X) \). Since \( p,q \in B_1^*(y, t, 16a, X) \), we have
\[
\begin{aligned}
d(\phi_s(y), \phi_s(p)) &< 16a \| X(\phi_s(y)) \| , \forall s \in [0, t], \\
d(\phi_s(y), \phi_s(q)) &< 16a \| X(\phi_s(y)) \| , \forall s \in [0, t] .
\end{aligned}
\]
Since \( 0<16a<c \), \( d(\phi_s(y), \phi_s(p)) < 16a \| X(\phi_s(y)) \|<c\|X(\phi_s(y))\| \), we have
\(
\| X(\phi_s(y)) \| \leq 2 \| X(\phi_s(p)) \|.\)
It follows that
\begin{align*}
d(\phi_s(y), \phi_s(p)) &< 16a \| X(\phi_s(y)) \| \leq 32a \| X(\phi_s(p)) \|, \\
d(\phi_s(y), \phi_s(q)) &< 16a \| X(\phi_s(y)) \| \leq 32a \| X(\phi_s(p)) \|.
\end{align*}

\noindent
Then
\[
\begin{aligned}
d(\phi_s(p), \phi_s(q)) &\leq d(\phi_s(p), \phi_s(y)) + d(\phi_s(y), \phi_s(q)) < 64a \| X(\phi_s(p)) \|.
\end{aligned}
\]
which contradicts that \(E\) is a rescaled \( 1\text{-}(t, 64a, K) \)-separating set. Thus, for each \( y \in G \), \( B_1^*(y, t, 16a, X) \) contains at most one point in \( E \).  The claim is proved.
\end{proof}

For any $x\in F$, denote by
$$E_x=\left\{q\in E :  \exists y\in F_x, s.t. \, q\in B_1^*(y, t, 16a, X)\right\}.$$
By Claim we know that $\#E_x\leq \#F_x \leq 3^{\left\lfloor \frac{t}{\tau} \right\rfloor}$ for any $x\in F$. By the fact that
$$E\subset \bigcup_{y\in G}B_1^*(y, t, 16a, X),$$
we know that $E\subset\bigcup\limits_{x\in F}E_x$.

Now we assume that  \(a>0\) is chosen small enough such that for any $y\in M\setminus {\rm Sing}(X)$ and any \( q \in B_1^*(y, t, 16a, X) \), one has
\[
\left| \int_0^t f(\phi_s(q)) ds - \int_0^t f(\phi_s(y)) ds \right| < \sigma t.
\]
For any $x\in F$ and any $q\in E_x$, assuming that $y\in F_x\subset B_2^*(x, t, \varepsilon, X)$ is chosen with $q\in B_1^*(y, t, 16a, X)$, by the choice of $0<\varepsilon<\varepsilon_2$, we have
\[
\frac{1}{t} \left| \int_0^t f(\phi_s(y)) ds - \int_0^t f(\phi_s(x)) ds \right| < \sigma,
\]
and then \[ \frac{1}{t}\left| \int_0^t f(\phi_s(q)) \, {\rm d}s - \int_0^t f(\phi_s(x)) \, {\rm d}s \right| < 2\sigma. \]

Now we have
\[
\begin{aligned}
\sum_{q \in E} e^{\int_0^t f(\phi_s (q)) \, {\rm d}s}  &= \sum_{x \in F} \sum_{q \in E_x} e^{\int_0^t f(\phi_s (q)) \, {\rm d}s}
\leq \sum_{x \in F} \sum_{q \in E_x} e^{2\sigma t} e^{\int_0^t f(\phi_s (x)) \, {\rm d}s}
\\&= e^{2\sigma t} \sum_{x \in F} \#E_x \cdot e^{\int_0^t f(\phi_s (x)) \, {\rm d}s }
\\ &\leq e^{2\sigma t} \cdot 3^{\frac{t}{\tau}} \sum_{x \in F} e^{\int_0^t f(\phi_s (x)) \, {\rm d}s}.
\end{aligned}
\]

\noindent
Take the supremum over all \( 1\text{-}(t, 64a, K) \)-separating set \(E\) and the infimum over all 2-$(t, \varepsilon, K)$ spanning set $F$, we have
\[
Z_{1,t}^*(X, f, 64a, K) \leq e^{2\sigma t} \cdot 3^{\frac{t}{\tau}} \cdot N_{2,t}^*(X, f, \varepsilon, K).
\]

\noindent
Thus
\[
\limsup_{t \to \infty} \frac{1}{t} \log Z_{1,t}^*(X, f, 64a, K) \leq \limsup_{t \to \infty} \frac{1}{t} \log N_{2,t}^*(X, f, \varepsilon, K) + 2\sigma + \frac{\log 3}{\tau}.
\]

\noindent
Let \( a \to 0 \) and \( \varepsilon \to 0 \), then
\[
\lim_{\varepsilon \to 0} \limsup_{t \to \infty} \frac{1}{t} \log Z_{1,t}^*(X, f, \varepsilon, K) \leq \lim_{\varepsilon \to 0}\limsup_{t \to \infty} \frac{1}{t} \log N_{2,t}^*(X, f, \varepsilon, K) + 2\sigma + \frac{\log 3}{\tau}.
\]

\noindent
By the arbitrariness of \( \tau \) and \( \sigma \), let \( \tau \to +\infty \) and \( \sigma \to 0 \), we have
\[
\lim_{\varepsilon \to 0} \limsup_{t \to \infty} \frac{1}{t} \log Z_{1,t}^*(X, f, \varepsilon, K) \leq \lim_{\varepsilon \to 0}\limsup_{t \to \infty} \frac{1}{t} \log N_{2,t}^*(X, f, \varepsilon, K).
\]
Taking the supremum over the compact subsets \( K \subset  M\setminus{\rm Sing}(X)\), we have
\[
\sup_K\lim_{\varepsilon \to 0} \limsup_{t \to \infty} \frac{1}{t} \log Z_{1,t}^*(X, f, \varepsilon, K) \leq\sup_K \lim_{\varepsilon \to 0} \limsup_{t \to \infty} \frac{1}{t} \log N_{2,t}^*(X, f, \varepsilon, K),
\]
and $Q_1^*(X, f)\leq P_2^*(X, f)$. The lemma is proved.
\end{proof}

By combining Lemma \ref{lemma5.1}, Lemma \ref{lemma5.4} and Lemma \ref{lemma5.5} we obtain Theorem \ref{theorem 1.3}. And Theorem \ref{theorem1.4:vp} is a direct corollary of the following lemma.

\begin{Lemma}
Let $\mu$ be a Borel probability measure with $\mu({\rm Sing}(X))=0$ and $f\in C(M, \mathbb{R})$. We have $P_\mu^*(X, f)\leq P_1^*(X, f)$.
\end{Lemma}
\begin{proof}
Fix $\sigma>0$. Then we can find $\delta>0$ such that
$$\lim_{\varepsilon\to 0}\limsup_{t\to\infty}\frac{1}{t}\log N_1^\mu(\delta, t, \varepsilon, X, f)>P_\mu^*(X, f)-\sigma.$$
Since \(\mu\) is regular and $\mu({\rm Sing(X)})=0$, there exists compact \(K\subset M\setminus{\rm Sing}(X)\) such that \(\mu(K)>1-\delta\). It is easy to see that
\[
\begin{aligned}
N_1^\mu(\delta, t, \varepsilon, X, f)
=&\inf\left\{\sum_{x\in F}e^{\int_0^t\phi_s(x){\rm d}x} : \mu\left(\bigcup_{x\in F}B_1^*(x, t, \varepsilon, X)\right)>1-\delta\right\} & \\ \leq&\inf\left\{\sum_{x\in F}e^{\int_0^t\phi_s(x){\rm d}x} : K\subset\bigcup_{x\in F}B_1^*(x, t, \varepsilon, X)\right\}
\\
=&N_{1, t}^*(X, f, \varepsilon, K).
\end{aligned}
\]

Then we have
$$P_\mu^*(X, f)-\sigma<\lim_{\varepsilon\to 0}\limsup_{t\to\infty}N_1^\mu(\delta, t, \varepsilon, X, f)\leq\lim_{\varepsilon\to 0}\limsup_{t\to\infty}N_{1, t}^*(X, f, \varepsilon, K)\leq P_1^*(X, f). $$
By letting $\sigma\to 0$, we get $P_\mu^*(X, f)\leq P_1^*(X, f)$.
\end{proof}

\section*{Appendix
 A. The metric pressure for fixed point free flow}
In the appendix, we improve Theorem 1.2 of  \cite{WangChen}. Here we consider a compact metric space $M$ and a continuous flow $\phi_t$ on $M$.   Inspired by the definition of entropy given by Thomas \cite{Thomas_1987} and Sun-Vargas \cite{Sun1999315}, Wang-Chen \cite{WangChen} used reparametrization Bowen balls defined as follows to define metric pressure. For any $x\in M$, $\varepsilon>0$ and $t>0$, let
$$\tilde B(x, t, \varepsilon)=\left\{y\in M : \exists \alpha\in {\rm Rep}[0,t], s.t. d(\phi_{\alpha(s)}(x), \phi_s(y))<\varepsilon, \forall\, 0\leq s\leq t\right.\}.$$
be {\it reparametrization Bowen ball} centered at $x$.
For any continuous function \( f \in C(X, \mathbb{R}) \), denote by
\[
\gamma_t(f, \varepsilon) = \sup \left\{ \left| \int_0^t \left( f(\phi_s (y)) - f(\phi_s (z)) \right) {\rm d}s \right|: y, z \in \tilde{B}(x, t, \varepsilon) \text{ for some } x \in M \right\}.
\]
If $f\in C(M, \mathbb{R})$ satisfies
\[
\lim_{\varepsilon \to 0} \limsup_{t \to \infty} \frac{\gamma_t(f, \varepsilon)}{t} = 0,
\]
then we say $f$ has {\it bounded variation}. In \cite{WangChen}, the authors prove that if $\phi_t$ has no fixed points and $f$ has bounded variation, then the following Katok's formula holds
$$\lim_{\varepsilon\to 0}\limsup_{t\to\infty}\frac{1}{t}\log\tilde{N}^\mu(\delta, t, \varepsilon, f)=h(\phi_1)+\int f{\rm d}\mu,$$
where $\delta\in(0,1)$ and
\[\tilde N^\mu(\delta, t, \varepsilon, f)=\inf_{}\left\{\sum_{x\in F}e^{\int_0^tf(\phi_s(x)){\rm d}s} \,:\, F \text{ is finite with }\mu\left(\bigcup_{x\in F}\tilde B(x, t, \varepsilon)\right)>1-\delta\right\}.\]
and the variation principle between topological pressure and metric pressure defined through reparametrization Bowen ball holds.

In the appendix, we show that the bounded variation of \(f\) is inherent to the function itself.

\bigskip
\noindent {\bf Proposition A.1. } {\it If $\phi_t$ has no fixed points, then every $f\in C(M, \mathbb{R})$ has bounded variation.}

\bigskip

Before proving the proposition, we collect the following result of Thomas \cite{Thomas_1987}.

\bigskip

\noindent\textbf{Lemma A.2. } {\it Let $\phi_t$ be a continuous flow on $M$ without fixed points. Given any \( \eta, r > 0 \), there exists \( \varepsilon_0 > 0 \) small enough such that for any \( x, y \in X \) and any \( \alpha \in \text{Rep}[0, t] \) (with \( t \geq r \)), if
\[
d(\phi_{\alpha(s)}(x), \phi_s(y)) < \varepsilon_0, \quad \forall \,0 \leq s \leq t,
\]
then
\[
|\alpha(s) - s| <
\begin{cases}
\eta s, &  r\leq s \leq t; \\
\eta r, &  s \leq r.
\end{cases}
\]
}
\bigskip
\noindent\textbf{Proof of Proposition A. 1. }
We prove that:  \(\forall \sigma > 0 \),  \( \exists\varepsilon_0 > 0 \) and \( T > 0 \) such that
$$\frac{1}{t} \left| \int_0^t \left( f(\phi_s (y)) - f(\phi_s (z)) \right) {\rm d}s \right| < \sigma$$
holds for any \( x \in M\), \( 0 < \varepsilon \leq \varepsilon_0 \), \( t > T \), and \( y, z \in \tilde{B}(x, t, \varepsilon) \).

Denote \( \|f\| = \sup\limits_{x \in M} |f(x)| \). Fix \( \sigma > 0 \), by the compactness of \( M\), we can take \( \eta > 0 \) such that if \( x, y \in M \) satisfy \( d(x, y) < \eta \), then \( |f(x) - f(y)| < \frac{\sigma}{8(\|f\| + 1)} \). Then we can take \( r > 0 \) such that for any \( x \in M \) and any \( y, z \in \phi_{[-2r, 2r]}(x) \), we have \( d(y, z) < \eta \). By Lemma A.2, there is \( 0 < \varepsilon_0 < \eta \) such that for any \( x, y \in M\) and any \( \alpha \in \text{Rep}[0, t] \) (with \( t \geq r \)), if
\[
d(\phi_{\alpha(s)}(x), \phi_s(y)) < \varepsilon_0, \quad \forall \,0 \leq s \leq t,
\]
then
\[
|\alpha(s) - s| <
\begin{cases}
\frac{\sigma}{8(\|f\| + 1)} s, & r\leq s \leq t ; \\
r, &  s \leq r.
\end{cases}
\]

Now we take \( 0 < \varepsilon \leq \varepsilon_0 \), \( t > r \) and \( y, z \in \tilde{B}(x, t, \varepsilon) \) for some \( x \in X \). Let \( n = \left\lfloor \frac{t}{r} \right\rfloor \). It is easy to see that
\[
\left| \int_0^t f(\phi_s (y)) {\rm d}s - \int_0^t f(\phi_s (z)) {\rm d}s - \left( \int_0^{nr} f(\phi_s (y)) {\rm d}s - \int_0^{nr} f(\phi_s (z)) {\rm d}s \right) \right| \leq 2 \|f\| r.
\]
In the following we give estimations on
\[ \left| \int_0^{nr} f(\phi_s (y)) {\rm d}s - \int_0^{nr} f(\phi_s (x)) {\rm d}s \right| \text{\ \ \ \ \ and\ \ \ \ } \left| \int_0^{nr} \varphi(\phi_s (y)) {\rm d}s - \int_0^{nr} \varphi(\phi_s (z)) {\rm d}s \right|  .\]  Firstly, we have
\[
\begin{aligned}
&\left| \int_0^{nr} f(\phi_s(y)) {\rm d}s - \int_0^{nr} f(\phi_{\alpha(s)}(x)) {\rm d}s \right| \leq \int_0^{nr} \left| f(\phi_s(y)) - f(\phi_{\alpha(s)}(x)) \right| {\rm d}s \\
<& \int_0^{nr} \frac{\sigma}{8(\|f\| + 1)} {\rm d}s \leq \frac{\sigma nr}{8(\|f\| + 1)} \leq \frac{\sigma t}{8(\|f\| + 1)}.
\end{aligned}
\]

\noindent
It is easy to see that
\[
\begin{aligned}
&\left| \int_0^{nr} f(\phi_{\alpha(s)}(x)) {\rm d}s - \int_0^{nr} f(\phi_s(x)) {\rm d}s \right| \\
\leq& \left| \int_0^{nr} f(\phi_{\alpha(s)}(x)) {\rm d}s - \int_0^{\alpha(nr)} f(\phi_s(x)) {\rm d}s \right| + \left| \int_0^{\alpha(nr)} f(\phi_s(x)) {\rm d}s - \int_0^{nr} f(\phi_s(x)) {\rm d}s \right| \\
\leq &\left| \int_0^{nr} f(\phi_{\alpha(s)}(x)) {\rm d}s - \int_0^{\alpha(nr)} f(\phi_s(x)) {\rm d}s \right| + \|f\| \cdot |\alpha(nr) - nr| \\
<& \left| \int_0^{nr} f(\phi_{\alpha(s)}(x)) {\rm d}s - \int_0^{\alpha(nr)} f(\phi_s(x)) {\rm d}s \right| + \frac{\sigma nr \|f\|}{8(\|f\| + 1)} \quad (\text{by Thomas's result}) \\
<& \left| \int_0^{nr} f(\phi_{\alpha(s)}(x)) {\rm d}s - \int_0^{\alpha(nr)} f(\phi_s(x)) {\rm d}s \right| + \frac{\sigma}{8} t.
\end{aligned}
\]

\noindent
Note that
\[
\begin{aligned}
&\left| \int_0^{nr} f(\phi_{\alpha(s)}(x)) {\rm d}s - \int_0^{\alpha(nr)} f(\phi_s(x)) {\rm d}s \right| \\
=& \left| \sum_{k=1}^n \int_{(k-1)r}^{kr} f(\phi_{\alpha(s)}(x)) {\rm d}s - \sum_{k=1}^n \int_{\alpha((k-1)r)}^{\alpha(kr)} f(\phi_s(x)) {\rm d}s \right| \\
=& \left| \sum_{k=1}^n f(\phi_{\alpha(s_k)}(x)) \cdot r - \sum_{k=1}^n f(\phi_{\alpha(s_k')}(x)) \cdot (\alpha(kr) - \alpha((k-1)r)) \right|,
\end{aligned}
\]
where \( s_k, s_k' \in [(k-1)r, kr] \).

Let arbitrary $1\leq k\leq n$ be given. Let \( \bar{x} = \phi_{\alpha((k - 1)r)}(x) \), \( \bar{y} = \phi_{(k - 1)r}(y) \),
\[
\bar{\alpha}(t) = \alpha(t + (k - 1)r) - \alpha((k - 1)r).
\]
then \( \bar{\alpha} \in \text{Rep}[0, t - (k - 1)r] \) and one can check that
\[
d(\phi_{\bar{\alpha}(s)}(\bar{x}), \phi_s(\bar{y}))
= d\left( \phi_{\alpha(s + (k - 1)r)}(x), \phi_{s + (k - 1)r}(y) \right) < \varepsilon \leq \varepsilon_0
\]
for any \( 0 \leq s \leq t - (k - 1)r \). By the choice of \( \varepsilon_0 \), we have
\[
\begin{aligned}
&|\alpha(s_k) - \alpha(s_k')|
=|\bar{\alpha}(s_k - (k - 1)r) - \bar{\alpha}(s_k' - (k - 1)r)| \\
\leq& |\bar{\alpha}(s_k - (k - 1)r) - (s_k - (k - 1)r)| + |s_k - s_k'| + |\bar{\alpha}(s_k' - (k - 1)r) - (s_k' - (k - 1)r)| \\
<& r + r + r < 4r.
\end{aligned}
\]
Note that
\[
|\bar{\alpha}(s) - s| < \begin{cases} \frac{\sigma}{8(\|f\| + 1)} s, & r\leq s \leq t; \\ r, & s \leq r. \end{cases}
\]
We also have
\[
|\alpha(kr) - \alpha((k - 1)r) - r| = |\bar{\alpha}(r) - r| < \frac{\sigma}{8(\|f\| + 1)} r.
\]

Note that $|\alpha(s_k)-\alpha(s_k')|< 4r$, by the choice of \( r \), we have \( d(\phi_{\alpha(s_k)}(x), \phi_{\alpha(s_k')}(x) ) < \eta \), and then
\[
\left| f( \phi_{\alpha(s_k)}(x) ) - f( \phi_{\alpha(s_k')}(x) ) \right| < \frac{\sigma}{8(\|f\| + 1)}, \quad \forall 1 \leq k \leq n.
\]
Hence, we have
\[
\begin{aligned}
&\left| \int_0^{nr} f( \phi_{\alpha(s)}(x) ) {\rm d}s - \int_0^{\alpha(nr)} f( \phi_s(x) ) {\rm d}s \right| \\
=& \left| \sum_{k=1}^n f( \phi_{\alpha(s_k)}(x) ) r - \sum_{k=1}^nf( \phi_{\alpha(s_k')}(x) ) \left( \alpha(kr) - \alpha((k - 1)r) \right) \right| \\
\leq &\left| \sum_{k=1}^n \left( f( \phi_{\alpha(s_k)}(x) ) - f( \phi_{\alpha(s_k')}(x) ) \right) r \right| +  \left| \sum_{k=1}^n \left(f( \phi_{\alpha(s_k')}(x) ) \left( \alpha(kr) - \alpha((k - 1)r) - r \right) \right)\right| \\
<& \frac{\sigma}{8(\|f\| + 1)} nr + \frac{\|f\| \sigma}{8(\|f\| + 1)} nr = \frac{\sigma}{8}nr \leq  \frac{\sigma}{8} t.
\end{aligned}
\]

\noindent
And then we have
\[
\begin{aligned}
&\left| \int_0^{nr} f( \phi_{\alpha(s)}(x) ) {\rm d}s - \int_0^{nr} f( \phi_s(x) ) {\rm d}s \right| \\
\leq& \left| \int_0^{nr} f( \phi_{\alpha(s)}(x) ) {\rm d}s - \int_0^{\alpha(nr)} f( \phi_s(x) ) {\rm d}s \right| + \left| \int_0^{\alpha(nr)} f\left( \phi_s(x) \right) {\rm d}s - \int_0^{nr} f\left( \phi_s(x) \right) {\rm d}s \right| \\
<& \frac{\sigma}{8} t + \frac{\sigma}{8} t = \frac{\sigma}{4} t.
\end{aligned}
\]

\noindent
Now we get
\[
\begin{aligned}
&\left| \int_0^{nr} f( \phi_s (y) ) {\rm d}s - \int_0^{nr} f( \phi_s(x) ) {\rm d}s \right|\\
\leq& \left| \int_0^{nr} f( \phi_s (y) ) {\rm d}s - \int_0^{nr} f( \phi_{\alpha(s)}(x) ) {\rm d}s \right| + \left| \int_0^{nr} f\left( \phi_{\alpha(s)}(x) \right) {\rm d}s - \int_0^{nr} f\left( \phi_s (x) \right) {\rm d}s \right| \\
<& \frac{\sigma}{8(\|f\| + 1)} t + \frac{\sigma}{4} t < \frac{3\sigma}{8} t.
\end{aligned}
\]

\noindent
Similarly, we have
\[
\left| \int_0^{nr} f( \phi_s (z) ) {\rm d}s - \int_0^{nr} f( \phi_s (x) ) {\rm d}s \right| < \frac{3}{8}\sigma t,
\]
then we have
\[
\left| \int_0^{nr} f( \phi_s (y) ) {\rm d}s - \int_0^{nr} f( \phi_s (z) ) {\rm d}s \right| < \frac{3}{4} \sigma t,
\]

\noindent
and then
\[
\left| \int_0^t f( \phi_s (y) ) {\rm d}s - \int_0^t f( \phi_s (z) ) {\rm d}s \right| < \frac{3}{4}\sigma t + 2 \|f\| r.
\]

\noindent
Choose \( T > r \) such that \( \frac{\sigma}{4} T \geq 2 \|f\| r \). Then we can see that if \( t > T \), then
\[
\left| \int_0^t f( \phi_s (y) ) {\rm d}s - \int_0^t f( \phi_s (z) ) {\rm d}s \right| < \frac{3}{4} \sigma t + 2 \|f\| r < \sigma t ,\quad 0 < \varepsilon \leq \varepsilon_0.
\]
This shows that if \( y, z \in \tilde{B}(x, t, \varepsilon) \) for some \( x \in M\) and \( t > T \), then
\[
\frac{1}{t} \left| \int_0^t f( \phi_s (y) ) {\rm d}s - \int_0^t f( \phi_s (z) ) {\rm d}s \right| < \sigma.
\]
The Lemma A.2 is proved. \hfill \qedsymbol

\bigskip

\bigskip
\noindent{\bf Lemma A.3. } {\it Let $\mu$ be a Borel probability measure on $M$ and $f\in C(M, \mathbb{R})$ and $\delta\in(0,1)$. We have
$$\lim_{\varepsilon \to 0} \limsup_{t \to \infty} \frac{1}{t} \log N^\mu(\delta, t, \varepsilon, f) = \lim_{\varepsilon \to 0} \limsup_{t \to \infty} \frac{1}{t} \log \tilde{N}^\mu(\delta, t, \varepsilon, f),$$
$$\lim_{\varepsilon \to 0} \liminf_{t \to \infty} \frac{1}{t} \log N^\mu(\delta, t, \varepsilon, f) = \lim_{\varepsilon \to 0} \liminf_{t \to \infty} \frac{1}{t} \log \tilde{N}^\mu(\delta, t, \varepsilon, f).$$}

\begin{proof}
It is easy to see that $B(x, t, \varepsilon)\subset \tilde{B}(x, t, \varepsilon)$ for any $x\in M$ and $t>0$ and $\varepsilon>0$. Then we can get that $N^\mu(\delta, t, \varepsilon, f)\geq \tilde{N}^\mu(\delta, t, \varepsilon, f)$. The we have
\[
\lim_{\varepsilon \to 0} \limsup_{t \to \infty} \frac{1}{t} \log N^\mu(\delta, t, \varepsilon, f) \geq \lim_{\varepsilon \to 0} \limsup_{t \to \infty} \frac{1}{t} \log \tilde{N}^\mu(\delta, t, \varepsilon, f),
\]
\[
\lim_{\varepsilon \to 0} \liminf_{t \to \infty} \frac{1}{t} \log N^\mu(\delta, t, \varepsilon, f) \geq \lim_{\varepsilon \to 0} \liminf_{t \to \infty} \frac{1}{t} \log \tilde{N}^\mu(\delta, t, \varepsilon, f).
\]

Now we fix an arbitrary $\sigma>0$ and $\tau>0$. It has been proved in \cite{Thomas_1987} that for any $a>0$, there is $\varepsilon>0$ such that for any $x\in M$, there is a subset $F_x\subset \tilde{B}(x, t, \varepsilon)$ with $\# F_x<3^{t/\tau+1}$ such that $\tilde{B}(x, t, \varepsilon)\subset \bigcup\limits_{y\in F_x}B(y, t, a)$. By Proposition A.1, we can choose $\varepsilon>0$ small enough with a constant $T>0$ such that for any $t>T$, $x\in M$ and any $y\in \tilde{B}(x, t, \varepsilon)$, one has
$$\frac{1}{t}\left|\int_0^tf(\phi_s(x)){\rm d}s-\int_0^tf(\phi_s(y)){\rm d}s\right|<\sigma.$$
Similar to Lemma \ref{lemma3.4}, we can get that
$$N^\mu(\delta, t, a, f)\leq 3^{t/\tau+1}e^{\sigma t}\tilde{N}^\mu(\delta, t, \varepsilon, f)$$
for all $t>\max\{T, \tau\}$. Then we have
$$\lim_{a\to 0}\limsup_{t\to\infty}\frac{1}{t}\log N^\mu(\delta, t, a, f)\leq \lim_{\varepsilon\to 0}\limsup_{t\to\infty}\frac{1}{t}\tilde{N}^\mu(\delta, t, \varepsilon, f)+\frac{\log 3}{\tau}+\sigma,$$
$$\lim_{a\to 0}\liminf_{t\to\infty}\frac{1}{t}\log N^\mu(\delta, t, a, f)\leq \lim_{\varepsilon\to 0}\liminf_{t\to\infty}\frac{1}{t}\tilde{N}^\mu(\delta, t, a, f)+\frac{\log 3}{\tau}+\sigma.$$
By the arbitrariness of $\tau$ and $\sigma$ we have
$$\lim_{a\to 0}\limsup_{t\to\infty}\frac{1}{t}\log N^\mu(\delta, t, a, f)\leq \lim_{\varepsilon\to 0}\limsup_{t\to\infty}\frac{1}{t}\tilde{N}^\mu(\delta, t, \varepsilon, f),$$
$$\lim_{a\to 0}\liminf_{t\to\infty}\frac{1}{t}\log N^\mu(\delta, t, a, f)\leq \lim_{\varepsilon\to 0}\liminf_{t\to\infty}\frac{1}{t}\tilde{N}^\mu(\delta, t, \varepsilon, f).$$
This finishes the proof of the lemma.
\end{proof}

\noindent
Then, by Lemma A.3 and Theorem 1 of \(\cite{he2004metrical}\), we have the following theorem:

\bigskip

\noindent
\textbf{Theorem A.4. }
{\it Let \(M\) be a compact metric space and $\phi_t$ be a fixed point free flow on $M$. For any $\phi$ ergodic invariant measure $\mu$ and any continuous function \(f \in C(X, \mathbb{R})\) and any $\delta\in(0,1)$, we have
\[
\lim_{\varepsilon \to 0} \limsup_{t \to \infty} \frac{1}{t} \log \tilde{N}^u(\delta, t, \varepsilon, f) = \lim_{\varepsilon \to 0} \liminf_{t \to \infty} \frac{1}{t} \log \tilde{N}^u(\delta, t, \varepsilon, f)=h_\mu(\phi_1) + \int f \, {\rm d}\mu.
\]
}

\end{document}